\documentclass{amsart}
\usepackage[noadjust]{cite}
\usepackage{amssymb}
\usepackage{xcolor}  
\definecolor{IthacaGreen}{RGB}{25,155,0}  
\definecolor{FallsBlue}{RGB}{43,112,226}

\usepackage{mathrsfs}
\usepackage{amsmath}
\usepackage{mathtools}
\usepackage{thmtools, thm-restate}
\usepackage{verbatim}
\usepackage{enumerate}
\usepackage{hyperref} 
    \usepackage{cleveref}
\usepackage{memhfixc} 
\hypersetup{ 
    colorlinks=true,
    linkcolor=IthacaGreen,    
    citecolor=IthacaGreen,
    urlcolor=IthacaGreen,
}
\usepackage{float}

\usepackage{tikz}
\usetikzlibrary{matrix}
\tikzset{diagram/.style={matrix of math nodes, inner sep=0pt, row
    sep=#1, column sep=2.5em, text height=1.5ex, text depth=.25ex,
    nodes={inner sep=1ex}}}
\tikzset{diagram/.default=2.5em}
\newcommand\diagram{\path node[diagram]}

\renewcommand\labelenumi{(\roman{enumi})}
\renewcommand\theenumi\labelenumi

\usepackage{memhfixc} 

\usepackage{amsthm}
\usepackage{thmtools, thm-restate}
\newtheorem{thm}{Theorem}[section]

\newtheorem{lem}[thm]{Lemma}
\newtheorem{cor}[thm]{Corollary}
\newtheorem{prop}[thm]{Proposition}
\newtheorem{defn}[thm]{Definition}

\newtheorem*{conjCeta}{Conjecture~$C_{\eta}$}

\newcommand{\ZZ}{\mathbb{Z}} 
\newcommand{\QQ}{\mathbb{Q}}

\newcommand{\PP}{\mathscr{P}}

\newcommand{\RR}{\mathbb{R}}

\newcommand{\FF}{\mathbb{F}}

\newcommand{\pp}{\mathfrak{p}} 
 
\newcommand{\bb}{\mathfrak{b}} 
 
\newcommand{\qq}{\mathfrak{q}} 
\newcommand{\aaa}{\mathfrak{a}} 
\newcommand{\mmm}{\mathfrak{m}}

\newcommand{\Stab}{\operatorname{Stab}}
\newcommand{\surj}{\twoheadrightarrow}
\newcommand{\inj}{\hookrightarrow}

\newcommand{\defeq }{\vcentcolon=}
\newcommand{\nequiv}{\not\equiv}
\newcommand{\ism}{\cong}

\newcommand{\inv}{^{-1}}

\newcommand{\MK}{\mathbf{M}_4}

\newcommand{\OO}{\mathcal{O}}
\newcommand{\ONE}{\textbf{1}}
\newcommand{\leg}[2]{\genfrac{(}{)}{}{}{#1}{#2}}
\newcommand{\mm}{\mathfrak{m}}

\newcommand{\DD}{\mathcal{D}}
\newcommand{\CL}{\mathrm{Cl}}

\newcommand{\parm}{\makebox[1.25ex]{\textbf{$\cdot$}}}
\DeclareMathOperator{\im}{im}
\DeclareMathOperator{\tr}{Tr}
\DeclareMathOperator{\GL}{GL}
\DeclareMathOperator{\spin}{spin}
\newcommand{\Norm}{\mathfrak{N}}
\DeclareMathOperator{\ord}{ord}
\DeclareMathOperator{\Gal}{Gal}

\DeclareMathOperator{\Art}{Art}

\newcommand{\preArt}[2]{\mathcal{A}_{#1}^{#2}}
\newcommand{\totpos}[1]{{#1}\succ 0}

    \newcommand{\switchAprint}[1]{}

\title {A Density of Ramified Primes}
\author{Stephanie Chan}
\address[Stephanie Chan]{Department of Mathematics,
University College London}
\email{\href{mailto:stephanie.chan.16@ucl.ac.uk}{stephanie.chan.16@ucl.ac.uk}}
\author{Christine McMeekin}
\address[Christine McMeekin]{Max-Planck-Institut f\"ur Mathematik
}
\email{\href{mailto:christine.mcmeekin@gmail.com}{christine.mcmeekin@gmail.com}}
\author{Djordjo Milovic}
\address[Djordjo Milovic]{Department of Mathematics,
University College London}
\email{\href{mailto:dzm656@gmail.com}{dzm656@gmail.com}}
\date{\today}

\begin{document}

\begin{abstract}
Let $K$ be a cyclic number field of odd degree over $\QQ$ with odd narrow class number, such that $2$ is inert in $K/\QQ$. We define a family of number fields $\{K(p)\}_p$, depending on $K$ and indexed by the rational primes $p$ that split completely in $K/\QQ$, in which $p$ is always ramified 
of degree $2$. Conditional on a standard conjecture on short character sums, the density of such rational primes $p$ that exhibit one of two possible ramified factorizations in $K(p)/\QQ$ is strictly between $0$ and $1$ and is given explicitly as a formula in terms of the degree of the extension $K/\QQ$. Our results are unconditional in the cubic case. Our proof relies on a detailed study of the joint distribution of spins of prime ideals.
\end{abstract}

\maketitle
\tableofcontents

\section{Introduction}\label{sec:intro}




Given a number field $K$, let $\OO$, $\CL$, and $\CL^+$ denote its ring of integers, its class group, and its narrow class group, respectively. We will prove certain density theorems for number fields $K$ satisfying the following conditions:
\begin{enumerate}
    \item[(C1)] $K/\QQ$ is Galois with cyclic Galois group;
    \item[(C2)] $[K:\QQ]$ is odd;
    \item[(C3)] $h^+ =|\CL^+|$ is odd;
    \item[(C4)] the prime $2$ is inert in $K/\QQ$.
\end{enumerate}

Recall that $\CL^+$ is the quotient of the group of invertible fractional ideals of $K$ by the subgroup of principal fractional ideals that can be generated by a totally positive element; in other words, $\CL^+$ is the ray class group of conductor equal to the product of all real places. If $\alpha\in K$ is totally positive, i.e., if $\sigma(\alpha)>0$ for all real embeddings $\sigma:K\hookrightarrow \RR$, we will sometimes write $\alpha\succ 0$. 
If $h^+$ is odd for a totally real field, then 
\begin{equation}\label{eqn:UT2}\switchAprint{eqn:UT2}
    \OO^{\times}_{+} := \{u\in\OO^{\times}: u\succ 0\} = \left(\OO^{\times}\right)^2.
\end{equation}
Notice that if $K/\QQ$ is an odd degree Galois extension, then $K/\QQ$ is totally real. Since $[\CL^+:\CL]$ is always a power of $2$, the condition (C3) implies that $\CL^+= \CL$.
Therefore the conditions that $K/\QQ$ is Galois, satisfying (C2) and (C3) together imply the following:
\begin{enumerate}
  \item[(P1)] $K/\QQ$ is Galois, $K$ is totally real, and $\CL^+ = \CL$.
\end{enumerate}
 If $K$ is totally real, then $\CL^+ = \CL$ if and only if every totally positive unit in $\OO$ is a square; see Lemma~\ref{RCFh}. Hence, property (P1) can be restated as
\begin{enumerate}
    \item[(P1)] $K/\QQ$ is Galois, $K$ is totally real, and $\OO^{\times}_{+} := \{u\in\OO^{\times}: u\succ 0\} = \left(\OO^{\times}\right)^2$.
\end{enumerate}

Number fields satisfying property (C1) and (P1) were studied by Friedlander, Iwaniec, Mazur, and Rubin \cite{FIMR}. More precisely, Friedlander et al.\ proved that if $\sigma$ is a (fixed) generator of $\Gal(K/\QQ)$, then the density of principal prime ideals $\pi\OO$ that split in the quadratic extension $K(\sqrt{\sigma(\pi)})/K$ is equal to $1/2$. Koymans and Milovic \cite{JDS} extended the results of Friedlander et al.\ in two different aspects. First, the number field $K$ now needs to satisfy only property (P1), i.e., $K/\QQ$ need not be cyclic; second, density theorems about the splitting behavior of principal prime ideals are proved for multi-quadratic extensions of the form $K(\{\sqrt{\sigma(\pi)}: \sigma\in S\})/K$, where $S$ is a fixed subset of $\Gal(K/\QQ)$ with the property that $\sigma\not \in S$ whenever $\sigma^{-1}\in S$.

Our main goal is to further extend these results to a certain setting where $S = \Gal(K/\QQ)\setminus\{1\}$; in this setting, we in fact have $\sigma\in S$ whenever $\sigma^{-1}\in S$, and so our work features a new interplay of the Chebotarev Density Theorem and the method of sums of type I and type II. In particular, the densities appearing in our main theorems are of greater complexity than those appearing in \cite{FIMR} or \cite{JDS}.

Another innovation in our work is that by assuming property (C3), we are now also able to study the splitting behavior of \textit{all} prime ideals, and not only those that are principal. While our generalization of ``spin'' to non-principal ideals may appear innocuous (see Definition~\ref{defn:spin}), it is of note that it still encodes the relevant splitting information as well as that the study of its oscillations requires new ideas, carried out in Section~\ref{sec:dFR}.

Let $K$ be a number field satisfying properties (P1) and (C3), and let $p$ be a rational prime that splits completely in $K/\QQ$. We will now define an extension $K(p)/\QQ$ where $p$ ramifies; this extension was first studied by McMeekin \cite{CMthesis}. 
Let $\pp$ be an unramified prime ideal of degree one in $\OO$. Let $R_{\pp}^{+}$ denote the maximal abelian extension of $K$ unramified at all finite primes other than $\pp$; in other words, $R_{\pp}^{+}$ is the ray class field of $K$ of conductor $\pp\infty$, where $\infty$ denotes the product of all real places of $K$. There is a unique subfield $K(\pp)\subset R_{\pp}^{+}$ of degree $2$ over $K$ when $\pp$ is prime to $2$ (see Lemma~\ref{Uquadsub}). Finally, we define $K(p)$ to be the compositum of $K(\pp)$ over all primes $\pp$ lying above $p$, i.e.,
$$
K(p) = \prod_{\pp|p}K(\pp).
$$
As $K(p)/\QQ$ is Galois, the residue field degree $f_{K(p)/\QQ}(p)$ of $p$ in $K(p)/\QQ$ is well-defined. Our goal is to study the distribution of $f_{K(p)/\QQ}(p)$ as $p$ varies.
Note that because $p$ splits completely in $K/\QQ$, $f_{K(p)/\QQ}(p)$ is equal to the residue field degree $f_{K(p)/K}(\pp)$ of $\pp$ in $K(p)/\QQ$ for any prime $\pp$ of $K$ lying above $p$. 
Furthermore, $f_{K(p)/\QQ}(p)=f_{K(p)/K}(\pp)$ must divide $2$ since $[K(p):K]$ is a power of $2$ and there are no cyclic subgroups of $\Gal(K(p)/K)$ of order greater than $2$.

To state our main results, we now introduce the relevant notions of density. For sets of primes $A\subseteq B$, we define the restricted density of $A$ (restricted to $B$) to be
\[
d(A|B):= \lim_{N\to\infty} \frac{\#A|_N}{\#B|_N}.
\]
where $A|_N:=\{p \in A: \Norm(p)<N\}$ and $B|_N$ is defined similarly. When $\Pi$ consists of all but finitely many primes, then $d(A):=d(A|\Pi)$ is the usual natural density of~$A$.

Let $\PP_\QQ^{2}$ denote the set of rational primes co-prime to $2$. For a fixed sign, $\pm$, we define the following sets of rational primes.
\[
\begin{array}{ l l}
 S & \defeq \{p \in \PP_\QQ^{2}: p \text{ splits completely in } K/\QQ\}, \\
 S_{\pm} & \defeq\{p\in S: p \equiv \pm 1 \bmod 4\ZZ\}, \\
 F & \defeq \{ p \in S: f_{K(p)/\QQ}(p) = 1\}, \\
 F_{\pm} & \defeq S_{\pm} \cap F. 
\end{array}
\]

Our main results are conditional on the following conjecture, a slight variant of which appears in both \cite{FIMR} and \cite{JDS}. In the following conjecture, the real number $\eta\in(0, 1]$ plays the role of $1/n$ from \cite[Conjecture~$C_n$, p.\ 738-739]{FIMR}.

\begin{conjCeta}[{\cite[Conjecture~$C_n$]{FIMR}}]
\label{Cm}\switchAprint{Cm}
Let $\eta$ be a real number satisfying $0<\eta\leq 1$. Then there exists a real number $\delta=\delta(\eta)>0$ such that for all $\epsilon>0$ there exists a real number $C=C(\eta, \epsilon)>0$ such that for all integers $Q\geq 3$, all real non-principal characters $\chi$ of conductor $q\leq Q$, all integers $N\leq Q^{\eta}$, and all integers $M$, we have  
\[
\left|\sum_{M<a\leq M+N}\chi(a)\right|\leq C Q^{\eta(1-\delta)+\epsilon}.
\]
\end{conjCeta}
We note that Conjecture~$C_{\eta}$ is known for $\eta>1/4$, as a consequence of the classical Burgess's inequality \cite{Burgess}, and remains open for $\eta\leq 1/4$. Moreover, for sums as above starting at $M = 0$, Conjecture~$C_{\eta}$ (for any $\eta$) is a consequence of the Grand Riemann Hypothesis for the $L$-function $L(s, \chi)$. We are now ready to state our main results.

\begin{restatable}{thm}{DRPformulas}\label{thm:DRP}\switchAprint{thm:DRP} Let $K$ be a number field satisfying conditions (C1)-(C4).
Assume Conjecture~$C_{\eta}$ holds for $\eta = \frac{2}{n(n-1)}$. For $k\neq 1$ dividing $n$ let $d_k$ be the order of $2$ in $(\ZZ/k)^\times$. Then for a fixed sign $\pm$,
\[
d(F_\pm|S_\pm)=\frac{s_\pm}{2^{3(n-1)/2}}, 
\quad \text{and} \quad d(F|S) = \frac{s_++s_-}{2^{(3n-1)/2}}
\]
where 
\[s_+=1+\prod_{\substack{k\mid n \\ d_k\text{odd}\\ k\neq 1}}2^{\frac{\phi(k)}{2d_k}}\left(\prod_{\substack{k\mid n \\ d_k\text{odd}\\ k\neq 1}}2^{\frac{\phi(k)}{2}}-1\right),\]
and
\[s_-=\prod_{\substack{k\mid n \\ d_k\text{even}\\ k\neq 1}}(2^{d_k/2}+1)^{\frac{\phi(k)}{d_k}}\prod_{\substack{k\mid n \\ d_k\text{odd}\\ k\neq 1}}(2^{d_k}-1)^{\frac{\phi(k)}{2d_k}},\]
where $\phi$ denotes the Euler's totient function. 
In particular, when $n$ is prime, writing $d=d_n$,
\[
(s_+,s_-)
=\begin{cases}
\left(1+2^{\frac{n-1}{2d}}(2^{\frac{n-1}{2}}-1),\ (2^d-1)^{\frac{n-1}{2d}}\right)
& \text{ if }d\text{ is odd,}\\
\hfil \left(1,\ (2^{\frac{d}{2}}+1)^{\frac{n-1}{d}}\right) & \text{ if }d\text{ is even.}
\end{cases}
\] 
\end{restatable}

The density $d(F|S)$ is determined by the product of densities $d(F|R)$ and $d(R|S)$ 
where $R$ is the set of primes satisfying a certain Hilbert symbol condition. Toward computing the density $d(R|S)$, the terms $s_\pm$ arise from counting the number of solutions to this Hilbert symbol condition over $(\OO/4)^\times/((\OO/4)^\times)^2$.

\begin{table}[h!]
\begingroup
\renewcommand{\arraystretch}{1.2} 
\label{tab:CompDensities}
\caption{Densities from Theorem \ref{thm:DRP}, computed for $K$ of degree $n$ satisfying the necessary hypotheses. 
\switchAprint{tab:CompDensities}}
\[
\begin{array}{c  | c  | c  | c }
n   & d(F_+|S_+) &  d(F_-|S_-) & d(F|S) \\ \hline
3    &  1/8  &    3/8  &    1/4   \\
5     &  1/64  &   5/64  &   3/64    \\
7     &  15/512  &    7/512  &    11/512    \\
9    &  1/4096  &   27/4096   &  7/2048    \\
11  &  1/32768  &    33/32768  &    17/32768    \\
13    &  1/262144  &   65/262144  &  33/262144    \\
15   & 1/2097152 & 375/2097152 & 47/262144 \\
\end{array}
\] 

\vspace{.2cm}

\endgroup
\end{table}

In the cubic case, we have the following unconditional theorem.

\begin{restatable}{thm}{DRPcubic}\label{thm:DRPcubic}\switchAprint{thm:DRPcubic}
Let $K/\QQ$ be a cubic cyclic number field and odd class number in which $2$ is inert. Then
\[
d(F|S)=\frac{1}{4},
\]
\[
d(F_+|S_+) = \frac{1}{8}, \quad \text{and} \quad
d(F_-|S_-) = \frac{3}{8}.
\]
\end{restatable}

For our main results, we have assumed that $K$ satisfies properties (C1)-(C4). To start, we need properties (P1) and (C3) to define the extensions $K(p)/K$ for primes $p$ that split completely in $K/\QQ$. Coincidentally, as mentioned above, property (C3) also allows us to study the splitting behavior of all (not necessarily principal) prime ideals. Property (C2) ensures that $\Gal(K/\QQ)$ contains no involutions. While methods to deal with involutions do exist (see \cite[Section 12, p.\ 745]{FIMR}), incorporating them into our arguments is non-trivial and may pose interesting new challenges in our analytic arguments. Properties (C1) and (C4) simplify our combinatorial arguments and allow us to give explicit density formulas. Removing the assumptions of properties (C1) and (C4) would pose new combinatorial challenges.    

To end this section, we give some examples of number fields satisfying (C1)-(C4) so as to convince the reader that our theorems are not vacuous. First, many such fields can be found within the parametric families given by Friedlander et al.\ in \cite[p.\ 712]{FIMR} and originally due to Shanks \cite{Shanks} and Lehmer \cite{Lehmer}, namely
$$
\{\QQ(\alpha_m):\ m\in\ZZ\}\quad \text{and}\quad \{\QQ(\beta_m):\ m\in\ZZ\}
$$
where $\alpha_m$ and $\beta_m$ are roots of the polynomials
$$
f_m(x) = x^3 + m x^2 + (m - 3) x - 1.
$$
and
\begin{multline*}
    g_m(x) = x^5+m^2x^4-2(m^3+3m^2+5m+5)x^3
    \\+(m^4+5m^3+11m^2+15m+5)x^2+(m^3+4m^2+10m+10)x+1,
\end{multline*}
respectively. While $\QQ(\alpha_m)$ and $\QQ(\beta_m)$ always satisfy properties (P1), (C1), and (C2), for small $m$ one can check for properties (C3) and (C4) using Sage or another similar mathematical software package. For instance, if $\beta_7$ is any root of 
$$
g_7(x) = x^5 + 49x^4 - 1060x^3 + 4765x^2 + 619x + 1,
$$
then $\QQ(\beta_7)$ is a totally real cyclic degree-$5$ number field of class number $1451$ where $2$ stays inert. We also note that one can use the law of cubic reciprocity to show that the fields $\QQ(\alpha_m)$ always satisfy property (C4).

More generally, we can look for special subfields of cyclotomic fields. Let $m$ be a prime number and $\zeta_{m}$ a primitive $m$-th root of unity, so that $\QQ(\zeta_{m})/\QQ$ is a cyclic extension of degree $\varphi(m)$, and suppose that $n$ is an odd integer such that $\varphi(m)\equiv 0\bmod 2n$. For instance, we can take $n$ to be a Sophie Germain prime and then take $m = 2n+1$ to also be a prime. Suppose also that $2$ is inert in $\QQ(\zeta_{m})$, i.e., that $2$ is a primitive root modulo $m$. We then define $K$ to be the unique subfield of $\QQ(\zeta_{m}+\zeta_{m}^{-1})$ of degree $n$ over $\QQ$; $K$ readily satisfies properties (C1), (C2), (C4), while for small $n$ the property that $\CL^+ = \CL$ and property (C3) can be checked using Sage. For instance, the unique degree-$5$ subfield of $\QQ(\zeta_{191})$ has class number $11$; it is isomorphic to $\QQ(\beta_2)$ with $\beta_2$ a root of the polynomial $g_2$ as above.

\section{Two Families of Number Fields}\label{sec:treasuremap}

We say a modulus $\mmm$ is \textit{narrow} whenever it is divisible by all real infinite places. We say a modulus is \textit{wide} whenever it is not divisible by any infinite place. We say a ray class group or ray class field is narrow or wide whenever its conductor is narrow or wide respectively. 

For $\mmm$ an ideal of $\OO$, let $\CL^+_\mmm$ denote the narrow ray class group of conductor $\mmm$. That is, $\CL^+_\mmm$ is the ray class group with conductor divisible by all real infinite places with finite part $\mmm$.


The following lemma leads to several equivalent formulations of property (P1).
\newpage

\begin{lem}\label{RCFh} \switchAprint{RCFh}
\begin{enumerate}[(1)] $K$ is any number field.
    \item The following are equivalent.
        \begin{enumerate}[(i)]
            \item $\CL^+=\CL$.
            \item Every principal ideal has a totally positive generator.
            \item All signatures are represented by units.
        \end{enumerate}
    \item If $h^+$ is odd, then $\CL^+=\CL$.
    \item $K$ is totally real with $\CL^+=\CL$ if and only if $\OO^{\times}_{+} = \left(\OO^{\times}\right)^2$.
\end{enumerate}
\end{lem}

Here, if $K$ is not necessarily totally real, an element is said to be totally positive when it is positive in all real embeddings, and the signature of an element is determined by the signs of the element in each real embedding. 

\begin{proof} Let $K$ be an arbitrary number field with $r_1$ real embeddings and $r_2$ pairs of complex embeddings.
That (i) and (ii) are equivalent follows from the definitions of the narrow and wide Hilbert class fields. 
By the exact sequence and canonical isomorphism in \cite[Theorem V.1.7]{MilneCFT} applied to the narrow modulus with trivial finite part, condition (i) is true exactly when $\OO^\times/\OO^\times_+ \ism (\ZZ/2)^{r_1}$. Noting that there are $r_1$ signatures and the signatures of two units are equal exactly when these units are equivalent modulo the totally positive units, (i) is equivalent to (iii).  
Since $[\CL^+:\CL]$ is always a power of $2$, if $h^+$ is odd, then property (i) holds.

As noted above, condition (i) is true exactly when $\OO^\times/\OO^\times_+ \ism (\ZZ/2)^{r_1}$. By Dirichlet's unit theorem, $\OO^\times/(\OO^\times)^2\ism \left(\ZZ/2\right)^{r_1+r_2}$.
Therefore if (i) holds and in addition $K$ is totally real, then $r_2=0$ and $\OO^\times/\OO^\times_+ \ism \OO^\times/(\OO^\times)^2$. Containment of $(\OO^\times)^2$ in $\OO^\times_+$ gives equality.
Conversely, if we assume $\OO^{\times}_{+} = \left(\OO^{\times}\right)^2$, then $\OO^\times/\OO^\times_+  \ism \left(\ZZ/2\right)^{r_1+r_2}$ by Dirichlet's unit theorem. By the exact sequence and canonical isomorphism in \cite[Theorem V.1.7]{MilneCFT}, there is an injection from
$\OO^\times/\OO^\times_+$ into a group isomorphic to $\left(\ZZ/2\right)^{r_1}$. Therefore $r_2=0$ and
$\OO^\times/\OO^\times_+ \ism \left(\ZZ/2\right)^{r_1}$. That is, $K$ is totally real and condition (i) holds.
\end{proof}


  \begin{lem}\label{Uquadsub}\switchAprint{Uquadsub} 
   Let $K$ be a totally real number field with odd narrow class number $h^+$. Let $\pp$ be an odd prime of $K$. Then the {narrow} ray class field over $K$ of conductor $\pp$ has a unique subextension that is quadratic over $K$. 
    \end{lem}
     \begin{proof}
 By \cite[V.1.7]{MilneCFT} applied to the modulus with finite part $\pp$ that is divisible by all real places, for any number field $K$, 
 \[
 h_\pp^+ = \frac{2^{n}(\Norm(\pp)-1)h}{(\OO^\times:\OO^\times_{\pp,1})}.
 \]
 Since $h_+$ is odd and $K$ is totally real, by Lemma \ref{RCFh}\switchAprint{RCFh}, $\OO^\times_+=(\OO^\times)^2$. Therefore
 \begin{align*}
     (\OO^\times:\OO^\times_{\pp,1}) &= (\OO^\times:\OO^\times_+)(\OO^\times_+:\OO^\times_{\pp,1}) \\
     &= (\OO^\times:(\OO^\times)^2)((\OO^\times)^2:\OO^\times_{\pp,1})\\
     &= 2^{n}((\OO^\times)^2:\OO^\times_{\pp,1}).
 \end{align*}
Therefore
  \[
 h_\pp^+ = \frac{(\Norm(\pp)-1)h}{((\OO^\times)^2:\OO^\times_{\pp,1})}.
 \]
 Also by \cite[V.1.7]{MilneCFT}, there is an injection
 \[
 \OO^\times/\OO^\times_{\pp,1} \inj \left(\ZZ/2\right)^{n} \times \left(\OO/\pp\right)^\times.
 \]
 The image of $(\OO^\times)^2/\OO^\times_{\pp,1}$ lies in $\left(\left(\OO/\pp\right)^\times\right)^2$, which has order $(\Norm(\pp)-1)/2$. Therefore $((\OO^\times)^2:\OO^\times_{\pp,1})$ divides $(\Norm(\pp)-1)/2$, so $h_\pp^+$ is even.
 
 Now we show the $2$-part of $\CL^+_\pp$ is cyclic. 
    Let $L$ denote the maximal $2$-extension of the narrow ray class field over $K$ of conductor $\pp$.
    Let $E$ denote the inertia group for $\pp$ relative to the extension $L/K$ and let $L_E$ denote the fixed field of $E$. 
    Since $\pp$ is odd, $\pp$ is tamely ramified in $L/K$, so by  \cite[7.59]{MilneANT}, $E$ is cyclic.
    
    By Lemma \ref{RCFh}\switchAprint{RCFh}, since $\OO^\times_+=(\OO^\times)^2$ and $h$ is odd, $\CL^+$ also has odd degree over $K$. Therefore there is no non-trivial even extension of $K$ in which all finite primes are unramified.
    All finite primes of $K$ are unramified in $L_E$ and $[L_E:K]$ divides $[L:K]$. Then since $[L:K]$ is a power of 2, $[L_E:K]=1$ so $E=\Gal(L/\QQ)$.
 \end{proof}
    
We may now define the multi-quadratic extension $K(p)/K$ as in Section~\ref{sec:intro}. In addition, we define another family of number fields parameterized by prime numbers $p$ for which our results also hold.
Consider a totally real number field $K$ with odd narrow class number $h^+$. Furthermore, we now impose the condition that $K/\QQ$ is a Galois extension. Equivalently, we are assuming conditions (P1) and (C3).

In Definition \ref{defn:Kp}\switchAprint{defn:Kp}, we apply Lemma \ref{Uquadsub}\switchAprint{Uquadsub} to ensure the existence of a unique quadratic subextension of the narrow ray class field over $K$ of conductor $\pp$. 
In Definition \ref{defn:Kpplus}\switchAprint{defn:Kpplus} we will use the fact that for such $K$, a principal ideal always has a totally positive generator; see Lemma~\ref{RCFh}.
   
\begin{defn}\label{defn:Kp}\switchAprint{defn:Kp}
Given an odd rational prime $p$ that splits completely in $K/\QQ$ and a prime ideal $\pp\subset\OO$ lying above $p$, define $K(\pp)$ to be the unique quadratic subextension of the narrow ray class field over $K$ of conductor $\pp$. 

Define $K(p)$ to be the compositum of the fields $K(\pp^{\sigma})$ as $\sigma$ ranges over $\Gal(K/\QQ)$.
\end{defn}

\begin{defn}\label{defn:Kpplus}\switchAprint{defn:Kpplus}
Given an odd rational prime $p$ that splits completely in $K/\QQ$, a prime ideal $\pp\subset\OO$ lying above $p$, and a totally positive generator $\alpha$ of the principal ideal $\pp^h$, we define
\[ K_+(\pp)\defeq K(\sqrt{\alpha}).\]
Define $K_+(p)$ to be the compositum of the number fields  $K_+(\pp^{\sigma})$ as $\sigma$ ranges over $\Gal(K/\QQ)$.
\end{defn}

Since $K$ is totally real and $h_+$ is odd, Lemma \ref{RCFh}\switchAprint{RCFh} implies that $\OO^{\times}_+ = (\OO^{\times})^2$, so $K_+(\pp)$ does not depend on the choice of totally positive generator $\alpha$. 

We note that while each of the fields $K(\pp^{\sigma})$ need not be Galois over $\QQ$, their compositum $K(p)$ certainly is. Similarly, $K_+(p)/\QQ$ is Galois, and each of the extensions $K_+(\pp^{\sigma})/\QQ$ need not be. 

As both extensions $K(p)/\QQ$ and $K_+(p)/\QQ$ are Galois, the ramification indices and residue field degrees of $p$ in each extension are well-defined. 
We denote the ramification index and residue field degree of $p$ in $K(p)/\QQ$ by $e_{K(p)/\QQ}(p)$ and $f_{K(p)/\QQ}(p)$, respectively. Similarly, we denote the ramification index and residue field degree of $p$ in $K_+(p)/\QQ$ by $e_{K_+(p)/\QQ}(p)$ and $f_{K_+(p)/\QQ}(p)$, respectively.

Since $p$ is assumed to split completely in $K/\QQ$, there are $n$ distinct primes in $K$ lying above $p$, and they are of the form $\pp^{\sigma}$, where $\pp$ is one such prime and $\sigma$ ranges over $\Gal(K/\QQ)$. 

By Lemma \ref{Uquadsub}\switchAprint{Uquadsub}, $K(\pp^\sigma)/K$ is a quadratic extension, and since $\alpha$ generates a prime ideal, $K_+(\pp^\sigma)/K$ is also a quadratic extension. 
Since $K(\pp^\sigma)$ is a subfield of the narrow ray class field over $K$ of conductor $\pp^\sigma$, the extension $K(\pp^\sigma)/K$ is unramified at $\pp^\tau$ for all $\tau\neq\sigma$ in $\Gal(K/\QQ)$. Since $h^+$ is odd, $K(\pp^\sigma)/K$ is ramified at $\pp^\sigma$. Therefore $e_{K(p)/\QQ}(p)=2$ and $[K(p):\QQ]=n2^n$.

Since $\pp^\sigma$ is an odd prime, $\pp^\sigma$ divides the discriminant of $K_+(\pp^\sigma)/K$ and so this extension is ramified at $\pp^\sigma$. 
Since $\pp^\tau$ does not divide the discriminant for any $\tau\neq\sigma$ in $\Gal(K/\QQ)$, $K_+(\pp^\sigma)/K$ is unramified at $\pp^\tau$ for all such $\tau$. Therefore $e_{K_+(p)/\QQ}(p)=2$ and $[K_+(p):\QQ]=n2^n$.

The residue field $\ZZ/p$ is cyclic and injects into $\mathcal{O}_{K(p)}/\mathfrak{P}$ where $\mathfrak{P}$ is a prime of $K(p)$ above $p$. Therefore $f_{K(p)/\QQ}(p) \mid 2$ because there are no cyclic subextensions of $K(p)/K$ of degree greater than $2$, and $p$ is assumed to split completely in $K/\QQ$. Similarly, $f_{K_+(p)/\QQ}(p) \mid 2$. We summarise in the following Lemma.

\begin{lem}\label{lem:possiblesplittings}\switchAprint{lem:possiblesplittings} 
Let $K$ be a totally real number field that is Galois over $\QQ$ with odd narrow class number and let $p$ be a prime that splits completely in $K/\QQ$. For $L=K(p)$ and for $L=K_+(p)$,
\begin{enumerate}
    \item $L/\QQ$ is a Galois extension of degree $n2^n$.
    \item $e_{L/\QQ}(p) = 2$. 
    \item $f_{L/\QQ}(p)\mid 2$.
\end{enumerate}
\end{lem}

We will see in Corollary~\ref{cor:SpinRes}\switchAprint{cor:SpinRes} that for a fixed odd rational prime $p$ splitting completely in $K/\QQ$, the residue field degrees of $p$ in $K(p)/\QQ$ and in $K_+(p)/\QQ$ are equal. 
Hence, to prove Theorem~\ref{thm:DRP}, we will prove the analogous results for the family of extensions $K_+(p)/\QQ$.

\section{The Spin of Prime Ideals}\label{sec:onespin}
 
Throughout this section, we will assume $K$ satisfies (P1) and (C3). By Lemma \ref{RCFh}\switchAprint{RCFh}, this is equivalent to assuming that $K$ is a totally real number field that is Galois over $\QQ$ with odd narrow class number, and these conditions imply that $\OO^{\times}_{+} =\left(\OO^{\times}\right)^2$. 
We give the following definition of \textit{spin}, which extends the definition of spin from \cite[(1.1)]{FIMR} in a natural way so that it applies to all odd ideals (not necessary principal).

\begin{defn}\label{defn:spin}\switchAprint{defn:spin} 
Let $\sigma \in \Gal(K/\QQ)$ be non-trivial. Given an odd ideal $\aaa$, we define the {spin} of $\aaa$ (with respect to $\sigma$) to be
\[
\text{spin}(\aaa,\sigma) = \left(  \frac{\alpha}{ \aaa^\sigma}   \right),
\]
where $\alpha$ is any totally positive generator of the principal ideal $\aaa^h$, and where $\left( \frac{\cdot}{\cdot} \right)$ denotes the quadratic residue symbol in $K$. 
\end{defn}


The assumption $\OO^{\times}_{+} =\left(\OO^{\times}\right)^2$ is important for two reasons. First, Lemma~\ref{RCFh} ensures that the principal ideal $\aaa^h$ has a generator $\alpha$ that is totally positive. Second, any two totally positive generators of $\aaa^h$ differ by a square, so the value of the quadratic residue symbol defining the spin does not depend on the choice of totally positive generator $\alpha$.

If $\aaa$ is an odd principal ideal and $\alpha_0$ is a totally positive generator of $\aaa$, then $\alpha_0^h$ is a totally positive generator for $\aaa^h$. As $h$ is odd, we have
$$
\left(  \frac{\alpha_0}{ \aaa^\sigma}   \right) = \left(  \frac{\alpha_0^h}{ \aaa^\sigma}   \right),
$$
so our definition coincides with that of Friedlander et al.\ in \cite{FIMR} for odd principal ideals $\aaa$.

\subsection{Known Results}
The main result in \cite{FIMR} can be stated as follows.

\begin{thm}[{\cite[Theorem 1.1]{FIMR}}]\label{thm:FIMR}\switchAprint{thm:FIMR} Suppose $K$ is a number field satisfying properties (P1) and (C1). Suppose $n = [K:\QQ]\geq 3$. Assume Conjecture~$C_{\eta}$ holds for $\eta = 1/n$ with $\delta = \delta(\eta)>0$. Let $\sigma$ be a generator of the Galois group $\Gal(K/\QQ)$. Then for all real numbers $x>3$, we have
\[
 \sum_{\substack{\pp\text{ principal} \\ \text{prime ideal} \\ \Norm(\pp) \leq x}}   \spin(\pp,\sigma)   \ll x^{1-\theta+\epsilon}
\]
where $\theta=\theta(n)=\frac{\delta}{2n(12n+1)}$. Here the implied constant depends only on $\epsilon$ and $K$.
\end{thm}
Friedlander et al.\ also proved an analogous result for the case when the summation is restricted to principal prime ideals $\pp$ with totally positive generators satisfying a suitable congruence condition.

By Burgess's inequality, Conjecture~$C_{\eta}$ holds for $\eta=1/3$ with $\delta= \frac{1}{48}$, so Theorem~\ref{thm:FIMR} holds unconditionally for $[K:\QQ]=3$ where $\theta= \frac{1}{10656}$. 

In \cite[Section 11]{FIMR}, Friedlander et al.\ pose some questions about the joint distribution of $\spin(\pp, \sigma)$ and $\spin(\pp, \tau)$ as $\pp$ varies over prime ideals, where $\sigma$ and $\tau$ are two distinct generators of the cyclic group $\Gal(K/\QQ)$. In \cite{JDS}, Koymans and Milovic prove that such spins are distributed independently if $n\geq 5$, i.e., that the product $\spin(\pp, \sigma)\spin(\pp, \tau)$ oscillates similarly as in Theorem~\ref{thm:FIMR}. In fact, they prove that the product of spins
$$
\prod_{\sigma\in S}\spin(\pp, \sigma)
$$
oscillates as long as the fixed non-empty subset $S$ of $\Gal(K/\QQ)$ satisfies the property that $\sigma\not\in S$ whenever $\sigma^{-1}\in S$. Moreover, their result holds for number fields $K$ satisfying property (P1) and having arbitrary Galois groups, i.e., not necessarily satisfying property (C1).

The assumption in \cite{JDS} that $\sigma\not\in S$ whenever $\sigma^{-1}\in S$ is made because $\spin(\pp, \sigma)$ and $\spin(\pp, \sigma^{-1})$ are not independent in the following sense. For a place $v$ of $K$, let $K_{v}$ denote the completion of $K$ at $v$. For $a,b\in K$ coprime to $v$, the Hilbert Symbol $(a, b)_v$ is defined to be $1$ if the equation $ax^2+by^2=z^2$ has a solution $x,y,z\in K_{v}$ with at least one of $x$, $y$, or $z$ non-zero and $-1$ otherwise.

\begin{prop}[{\cite[Lemma~11.1]{FIMR}}]\label{prop:spinrelation}\switchAprint{prop:spinrelation}  Suppose $K$ is a number field satisfying properties (P1) and (C3). Suppose $\pp\subset \OO$ is a prime ideal and $\sigma\in\Gal(K/\QQ)$ is an automorphism such that  $\pp$ and $\pp^\sigma$ are relatively prime. Then
\[
\spin(\pp,\sigma)\spin(\pp,\sigma\inv) = \prod_{v|2} (\alpha,\alpha^\sigma)_v,
\]
where $\alpha$ is a totally positive generator of $\pp^h$ and the product is taken over places $v$ dividing $2$.
\end{prop}

\begin{proof} 
This is essentially Lemma 11.1 in \cite{FIMR}. 
The proof uses the fact that
\[
\prod_{v}(\alpha,\alpha^\sigma)_v=1.
\]
Since $\alpha\succ 0$, $(\alpha,\alpha^\sigma)_v=1$ for all infinite places $v$.


 Consider $v$, a finite place not equal to $\pp$ or $\pp^\sigma$, and not dividing $2$.
 Since $v\neq \pp,\pp^\sigma$, we have $\alpha$ and $\alpha^\sigma$ are non-zero modulo $v$. Consider the equation 
\[
	\alpha^\sigma x^2 \equiv 1 - \alpha y^2 \bmod v.
	\]
	The right hand side and the left hand side each take on $(\Norm(v)+1)/2$ values, so there is a solution by the pigeon hole principle. 
	It can not be the case that both $x$ and $y$ are $0$. 
	Suppose $x\nequiv 0 \bmod v$. 
	Since $v$ is prime to 2 and $x\nequiv 0$, Hensel's Lemma implies there exists a solution in the completion at $v$. Therefore $(\alpha, \alpha^\sigma)_v=1$.
	If $y$ is non-zero, a similar argument works. 

Since $\alpha$ and $\alpha^\sigma$ are relatively prime, $(\alpha,\alpha^\sigma)_\pp = \text{spin}(\pp,\sigma\inv)$ and $(\alpha,\alpha^\sigma)_{\pp^\sigma} = \text{spin}(\pp,\sigma)$. 
Then since $\prod_{v}(\alpha,\alpha^\sigma)_v=1$, we are done.
\end{proof}

In this paper, we study the joint distribution of multiple spins $\spin(\pp, \sigma)$, $\sigma\in S$, in a setting where there are in fact many $\sigma\in S$ such that $\sigma\inv\in S$ as well. From the discussion above, we see that this might involve combining the work of Koymans and Milovic with the study of the products $\spin(\pp,\sigma)\spin(\pp,\sigma\inv)$ for various $\sigma$.

\subsection{Factorization and Spin}

The spin of prime ideals is related to the splitting behavior of $p$ in both $K_+(p)$ and $K(p)$ as we will see in Proposition \ref{prop:splittingandspin}\switchAprint{prop:splittingandspin} and Corollary \ref{cor:SpinRes}\switchAprint{cor:SpinRes}.

Let $R_\mmm^+$ denote the narrow ray class field over $K$ of conductor $\mmm$. Let $\pp$ be an odd prime of $K$. Recall from Definition~\ref{defn:Kp} that Lemma \ref{Uquadsub}\switchAprint{Uquadsub} gives the existence of a unique quadratic subextension of $R_{\pp}^+/K$, denoted by $K(\pp)$. We have the following proposition for $K$ satisfying properties (P1) and (C3).

\begin{prop}\label{treasuremap}\switchAprint{treasuremap} Suppose $K$ is a totally real number field with odd narrow class number, and suppose $K/\QQ$ is a Galois extension. Let $\pp\subset \OO$ be an odd prime ideal. Let $\alpha\in\OO$ be a totally positive generator of $\pp^h$. Then $$
K(\pp) = K(\sqrt{u\alpha})
$$
for some unit $u\in\OO^\times$ well-defined modulo $(\OO^\times)^2$. We denote the unit class of $u$ by $\mathbf{u}_K(\pp)\in\OO^\times/(\OO^\times)^2$. Furthermore, $\mathbf{u}_K(\pp^\sigma)=\mathbf{u}_K(\pp)^\sigma$ for any $\sigma\in \Gal(K/\QQ)$.
\end{prop}
\begin{proof}
By \cite[3.10.3 and 3.10.4]{HelmutKoch}, and since $\pp$ is totally ramified in $K(\pp)/K$ as in \ref{lem:possiblesplittings}\switchAprint{lem:possiblesplittings},
$
K(\pp)=K(\gamma)
$
for some integral $\gamma\in K(\pp)$ with minimal polynomial over $K$ Eisenstein at $\pp$.
If $f_\gamma(x)=x^2 + c_1 x + c_0\in \OO[x]$ is the minimal polynomial of $\gamma$ over $K$, we could take $\gamma'\defeq  2\gamma + c_1\in K(\pp)$ instead as a primitive element and the minimal polynomial of $\gamma'$ is $x^2 - (c_1^2-4c_0)$, which is also Eisenstein at $\pp$ since $\pp$ is odd. Therefore we may assume that the minimal polynomial of $\gamma$ over $K$ takes the form
$
f_\gamma(x) = x^2 - c
$
for some $c\in \OO$ where $f_\gamma(x)$ is Eisenstein at $\pp$.

If $(c)$ had a prime factor $\qq\neq \pp$ with odd multiplicity, then $K(\pp)$ would be ramified at $\qq$, which is impossible as $K(\pp)\subseteq R^+_{\pp}$. Therefore
$
(c) = \pp I^2
$
for some ideal $I\subseteq \OO$ coprime to $\pp$.
Let $b\in\OO$ be a generator of $I^{h}$. Let $\alpha\in\OO$ be a totally positive generator of $\pp^h$, which exists by Lemma \ref{RCFh}\switchAprint{RCFh}. Raising to the power of $h$ gives $(c)^h=(\alpha)(b)^2$.
Since $h$ is odd, we can write
\begin{equation}
\label{eq:tm1}
u\alpha = c \left( \frac{c^{({h-1})/{2}}}{b} \right)^2
\end{equation}
for some unit $u\in\OO^\times$. Therefore 
$$
K(\sqrt{u\alpha}) = K(\sqrt{c})=K(\gamma)=K(\pp).
$$
As $K(\pp)\subset R_{\pp}^+$, we note that $u\alpha$ is a square in $R_\pp^+$.

If $K(\sqrt{u\alpha})$ and $K(\sqrt{v\alpha})$ are both contained in $R_\pp^+$ for $u,v\in \OO^\times$, then $K(\sqrt{u\alpha})=K(\sqrt{v\alpha})$ by uniqueness of the quadratic subextension coming from Lemma \ref{Uquadsub}\switchAprint{Uquadsub}. Then we can write 
\[
\sqrt{u\alpha} = r_1 + r_2\sqrt{v\alpha}
\]
for some $r_1,r_2\in K$. Then $u\alpha = r_1^2 +2r_1r_2\sqrt{v\alpha}+v\alpha r_2^2$ so one of $r_1$ or $r_2$ must be $0$. Since $u\alpha$ generates a prime to an odd power, $u\alpha$ cannot be a square in $K$ so $r_2\neq 0$. Then $r_1=0$ so $u\alpha= v\alpha r_2^2$. Therefore $r_2$ is a unit and $u$ and $v$ represent the same class in ${\OO^\times/(\OO^\times)^2}$.

It remains to show that $\mathbf{u}_K(\pp^\sigma)=\mathbf{u}_K(\pp)^\sigma$ for any $\sigma\in \Gal(K/\QQ)$.  It suffices to show that $u^{\sigma}$ is in the class $\mathbf{u}_K(\pp^{\sigma})$ for $u$ a representative of $\mathbf{u}_K(\pp)$. Since $\alpha^{\sigma}$ is a totally positive generator of $(\pp^{\sigma})^h$, it suffices to prove that $K(\pp^{\sigma}) = K\left(\sqrt{(u\alpha)^{\sigma}}\right)$.

Let $u\in\OO^\times$ be a representative of $\mathbf{u}_K(\pp)$. Let $L=K\left(\sqrt{u\alpha}\right)$ so that $L=K(\pp)$, and let $L_\sigma = K\left(\sqrt{(u\alpha)^\sigma}\right)$. 
It is clear that $\pp^\sigma$ ramifies in $L_\sigma/K$. Suppose an arbitrary prime $\qq$ ramifies in $L_\sigma/K$. Then $L_\sigma=K\left(\sqrt{d}\right)$ for some $d\in \OO$ such that $x^2-d$ is Eisenstein at $\qq$ and $(u\alpha)^\sigma=r^2d$ for some $r\in K$. Then 
\[
u\alpha=\left(r^{\sigma\inv}\right)^2d^{\sigma\inv},
\]
so $L=K\left(\sqrt{d^{\sigma\inv}}\right)$.
Since $x^2-d$ is Eisenstein at $\qq$, $x^2-d^{\sigma\inv}$ is Eisenstein at $\qq^{\sigma\inv}$, so $\qq^{\sigma\inv}$ is ramified in $L/K$. Therefore if any prime except for $\pp^\sigma$ were to ramify in $L_\sigma/K$, this would contradict that $L\subseteq R_\pp^+$, proving that $L_\sigma=K(\pp^\sigma)$.
\end{proof}





\begin{lem}\label{QR}\switchAprint{QR}
Suppose $K$ is a number field and $h$ is an odd number. Suppose $\aaa$ and $\bb$ are distinct odd primes of $K$, and suppose $\alpha$ and $\beta$ are totally positive generators of $\aaa^h$ and $\bb^h$, respectively, such that any prime above $2$ is unramified in $K(\sqrt{\beta})/K$. Then
\[
\left( \frac{\alpha}{\bb} \right)=\left( \frac{\beta}{\aaa} \right),
\]
where $(\cdot/\cdot)$ denotes the quadratic residue symbol in $K$.
\end{lem}

\begin{proof}
Since $h$ is odd and $\bb$ and $\aaa$ are coprime, we have
\begin{equation}\label{eq:QR:2}
\left( \frac{\alpha}{\bb} \right) = \left( \frac{\alpha}{\bb} \right)^h = \left( \frac{\alpha}{\bb^h} \right) = \left( \frac{\alpha}{\beta} \right).
\end{equation}
Similarly,
\begin{equation}\label{eq:QR:1}
\left( \frac{\beta}{\aaa} \right) = \left( \frac{\beta}{\aaa} \right)^h = \left( \frac{\beta}{\aaa^h} \right) = \left( \frac{\beta}{\alpha} \right).
\end{equation}
By the law of quadratic reciprocity for $K$ \cite[Theorem VI.8.3, p.\ 415]{NeukirchANT}, we have
$$
\left( \frac{\alpha}{\beta} \right)=\left( \frac{\beta}{\alpha} \right) \prod_{v|2\infty}(\alpha,\beta)_v,
$$
where $(\cdot, \cdot)_v$ is the Hilbert symbol on $K$ and the product above is over all places $v$ lying above $2$ and infinity.
 
For each infinite place $v$, we have $(\alpha,\beta)_v = 1$ since $\alpha$ is totally positive (and thus also positive in the embedding of $K$ into $\mathbb{R}$ corresponding to $v$).
For any place $v$ lying above $2$, we have $(\alpha,\beta)_v = 1$ since $\alpha$ is coprime to $2$ and any even prime is unramified in $K(\sqrt{\beta})/K$.
We thus deduce that
$$
\left( \frac{\beta}{\alpha} \right) = \left( \frac{\alpha}{\beta} \right),
$$
which in combination with \eqref{eq:QR:1} and \eqref{eq:QR:2} yields the desired result.
\end{proof}

Given a rational prime $p$, fix a prime $\pp$ above $p$ and a totally positive generator $\alpha$ of $\pp^h$.
Recall from Definition \ref{defn:Kpplus}\switchAprint{defn:Kpplus} that  $K_+(p)$ is the composite of $K_+(\pp^{\sigma})$ as $\sigma$ varies over all elements of $\Gal(K/\QQ)$, where $K_+(\pp^{\sigma})\defeq K(\sqrt{\alpha^\sigma})$. As before, denote by $K(\pp^{\sigma})$ the unique quadratic subextension of the narrow ray class field over $K$ of conductor $\pp^{\sigma}$. 

 The factorization of $p$ in $K_+(p)$ or $K(p)$ is determined by the factorizations of $\pp$ in each $K_+(\pp^{\sigma})$ or $K(\pp^{\sigma})$ respectively, which is in turn determined by the spin of $\pp$ with respect to $\sigma$ or $\sigma^{-1}$, respectively.

For an abelian extension of number fields $L/E$ and a prime $\pp$ of $E$, let $f_{L/E}(\pp)$ denote the residue field degree of $\pp$ in $L/E$.

\begin{prop}\label{prop:splittingandspin}\switchAprint{prop:splittingandspin}
Assume $K$ satisfies properties (P1) and (C3). 
For a fixed odd prime $\pp$ of $K$ that splits completely in $K/\QQ$ and $\sigma$ non-trivial in $\Gal(K/\QQ)$, the following are equivalent.
\begin{enumerate}
    \item $\spin(\pp,\sigma) = 1 $,
    \item $f_{K(\pp^{\sigma})/K}(\pp)=1$,
    \item $f_{K_+(\pp^{\sigma\inv})/K}(\pp)=1$.
\end{enumerate}
\end{prop}

\begin{proof}
For $\gamma\in\OO$, $L\defeq K(\sqrt{\gamma})$, and $\qq$ any prime of $K$, $f_{L/K}(\qq)=1$ if and only if $\gamma$ is a square modulo $\qq$ because the natural injective homomorphism of residue class fields is surjective exactly when $\sqrt{\gamma}$ has a pre-image. If $\qq$ is unramified in $L/K$, then $f_{L/K}(\qq)=1$ exactly when the quadratic residue $(\gamma/\qq)=1$.

Take $\pp$ to be an odd prime of $K$ splitting completely in $K/\QQ$, take $\alpha$ to be a totally positive generator of $\pp^h$, and take $\sigma\in\Gal(K/\QQ)$ non-trivial. Then $\pp$ is unramified in $K_+(\pp^{\sigma\inv})/K$ so $f_{K_+(\pp^{\sigma\inv})/K}(\pp)=1$ if and only if
\[
\left(\frac{\alpha^{\sigma\inv}}{\pp}\right)= \left(\frac{\alpha}{\pp^\sigma}\right)=\spin(\pp,\sigma)=1.
\]

By Proposition \ref{treasuremap}\switchAprint{treasuremap},  $K(\pp^{\sigma}) =K(\sqrt{(u\alpha)^\sigma})$ where $u$ is in the unit class $\mathbf{u}_K(\pp)$. Then since $\pp$ is unramified, $f_{K(\pp^{\sigma})/K}(\pp)=1$ if and only if
\[
\left(\frac{(u\alpha)^\sigma}{\pp}\right)=1.
\]
Lemma \ref{QR}\switchAprint{QR} applies because $\pp$ and $\pp^\sigma$ are co-prime and primes above $2$ are unramified in $K(\sqrt{(u\alpha)^\sigma})/K$ since $K(\sqrt{(u\alpha)^\sigma})\subseteq R_{\pp^\sigma}^+$. Therefore
 $
 \left({(u\alpha)^\sigma}/{\pp}\right)=\left({\alpha}/{\pp^\sigma}\right),
 $
 so $f_{K(\pp^{\sigma})/K}(\pp)=1$ if and only if
 \[
\left(\frac{\alpha}{\pp^\sigma}\right)=\spin(\pp,\sigma)=1.\qedhere
\]
\end{proof}

\begin{cor}\label{cor:SpinRes}\switchAprint{cor:SpinRes}
For a fixed odd rational prime $p$ splitting completely in $K/\QQ$, the residue field degrees of $p$ in the extensions $K(p)/\QQ$ and $K_+(p)/\QQ$ are equal to $1$ if and only if $\spin(\pp,\sigma)=1$ for all non-trivial $\sigma\in\Gal(K/\QQ)$. Otherwise these residue field degrees are equal to $2$.
\end{cor}
\begin{proof}
$f_{K(p)/\QQ}(p)=1$ exactly when $f_{K(\pp^{\sigma})/K}(\pp)=1$ for all $\sigma\in\Gal(K/\QQ)$. Similarly, $f_{K_+(p)/\QQ}(p)=1$ exactly when $f_{K_+(\pp^{\sigma\inv})/K}(\pp)=1$ for all $\sigma\in\Gal(K/\QQ)$. Apply Proposition \ref{prop:splittingandspin}\switchAprint{prop:splittingandspin}. If the residue field degrees are not equal to $1$ then they are equal to $2$ by Lemma \ref{lem:possiblesplittings}.
\end{proof}

\section{A Consequence of the Chebotarev Density Theorem}\label{sec:dRS}

In this section, we use the Chebotarev Density Theorem to prove that the primes of $K$ are equidistributed in $\mathbf{M}_4$ as defined below, where the mapping takes primes to a totally positive generator considered in $\mathbf{M}_4$. This contributes toward the density $d(R|S)$ of rational primes $p$ that 
satisfy the spin relation,
\[
\spin(\pp,\sigma) \spin(\pp,\sigma\inv)=1 \quad \text{for all non-trivial } \sigma\in\Gal(K/\QQ), 
\]
 where $\pp$ is a prime of $K$ above $p$, restricted to the rational primes splitting completely in $K/\QQ$. We will also give this density restricted modulo $4\ZZ$. Theorem \ref{thm:pmdensityformulas}\switchAprint{thm:pmdensityformulas} and Proposition \ref{prop:Starlight}\switchAprint{prop:Starlight} together give the density of such primes satisfying the spin relation.
 
 \begin{defn}\label{defn:MK}\switchAprint{defn:MK}
For $q$ a power of $2$, define
\[
\mathbf{M}_q:= (\OO/q\OO)^\times / \left( (\OO/q\OO)^\times \right)^2.
\]
Note that $\mathbf{M}_q$ is a group with a natural action from $\Gal(K/\QQ)$.
\end{defn}

\begin{prop}\label{prop:Mundy:ext}\switchAprint{prop:Mundy:ext}
Let $K$ be a number field  satisifying (C1) and (C4). Then
\begin{enumerate}
    \item $\MK \ism (\ZZ/2)^n$ as a $\ZZ/2$-vector space, 
    \item the invariants of the action of $\Gal(K/\QQ)$ on $\MK$ are exactly $\pm 1$.
\end{enumerate}
\end{prop}

\begin{proof} Let $U_m\defeq (\OO/m)^\times$. 
\begin{enumerate}
    \item Fix a set of representatives $\mathcal{R}$ for $\OO/2$ in $\OO$. Let $\mathcal{R}^\times$ be a subset of $\mathcal{R}$ containing representatives for $(\OO/2)^\times$. Observe that $\{ x+2y: x \in\mathcal{R}^\times, y \in \mathcal{R} \}$ is a set of representatives for $U_4$ and $\#U_4 = 2^n(2^n-1)$. Therefore elements of $U_4^2$ are of the form $(x+2y)^2 \equiv x^2 \bmod 4\OO$ for $x \in\mathcal{R}^\times$ and $y \in \mathcal{R}$. Since $\#(\OO/2)^\times = 2^n -1$ is odd, the squaring map on $U_2=(\OO/2)^\times$ is surjective and so $\#U_4^2 = 2^n-1$. Therefore $\#\MK = \#U_4/\#U_4^2 = 2^n$. Since $\MK$ is formed by taking the quotient of $U_4$ modulo squares, $\MK$ is a direct product of cyclic groups of order $2$.
    
    For any $\alpha\in \OO$ coprime to $2$, write $[\alpha]$ as the projection of $\alpha\OO$ in $\MK$.
    Since every $x\in\mathcal{R}^\times$ is a square in $U_2$, we can write down the isomorphism explicitly as
    \begin{equation}\label{eq:M4iso}
\mathbf{M}_4\rightarrow \OO/2\cong \FF_{2^n}\qquad
[x+2y]=[1+2x^{-1}y]\mapsto x^{-1}y.
\end{equation}
 We see that $\mathbf{M}_4=\{[1+2y]:y\in \OO/2\}$.

    \item Let $\sigma$ be a generator of $\Gal(K/\QQ)$. The action of $\sigma$ on $[1+2y]\in\mathbf{M}_4$, simply maps $y$ to $y^{\sigma}$. Then we see that $y\equiv y^\sigma\bmod \OO/2$ if and only if $y\equiv 0$ or $1\bmod \OO/2$. These correspond to $\pm 1$ in $\MK$.\qedhere

\end{enumerate}
\end{proof}
 

\begin{lem}\label{lem:Hilb.welldefn.M4}\switchAprint{lem:HilSymM4welldefn}
Let $K$ be a number field such that $2$ is inert in $K/\QQ$. The Hilbert symbol $(\parm, \parm)_2$ is well-defined on $\mathbf{M}_4$.
\end{lem}
\begin{proof}
We show that $(\alpha,\beta)_2=(\alpha+4B,\beta)_2$ for any $B\in\OO$ coprime to $2$, which implies that $(\parm, \parm)_2$ is well-defined on $(\OO/4\OO)^{\times}\times(\OO/4\OO)^{\times}$.
Suppose $B\in\OO$ is coprime to $2$.
It suffices to show that $(\alpha,\beta)_2=1$ implies $(\alpha+4B,\beta)_2=1$.
Take $x,y,z\in \OO$ not all divisible by $2$ satisfying
$x^2-\alpha y^2=\beta z^2\bmod 8$.
 Since $(\OO/2)^{\times}$ contains all its squares,
 there exists $C,D\in \OO$ such that
 $C^2\equiv \alpha^{-1}\beta B\bmod 2$ and 
 $D^2\equiv \alpha^{-1}\beta^{-1}B\bmod 2$.
Take $X=x+2Cz$, $y=Y$ and $Z=z+2Dx$, then one can check that $X^2-(\alpha+4B) Y^2\equiv\beta Z^2\bmod 8$.
\end{proof}

\begin{lem}\label{lem:Hilb.nondeg.M4}\switchAprint{lem:Hilb.nondeg.M4} Let $K$ be a number field such that $2$ is inert in $K/\QQ$.
The Hilbert symbol $(\parm,\parm)_2$ is non-degenerate on $\mathbf{M}_4$.
\end{lem}
\begin{proof}
Fix some $\alpha\in \OO$ coprime to $2$.
We claim that $(\alpha+4B,2)_2=1$ for some $B\in\OO$.
Since $(\OO/2)^\times$ contains all its squareroots, there exist some $\gamma, z\in \OO$ such that $\alpha\equiv\gamma^2-2z^2\bmod 4$.
Write $x=\gamma+2x'$ for some $x'\in\OO$, set $B=x'\gamma+x'^2$ and $y=1$.
Then $x^2-(\alpha+4B)y^2\equiv 2z^2\bmod 8$. This proves our claim.

Now suppose $(\alpha,\beta)_2=1$ for all $\beta\in\OO$ coprime to $2$. Then taking $B$ from the above claim, $(\alpha+4B,\beta)_2=1$ holds for all $\beta\in\OO$ coprime to $2$ by Lemma~\ref{lem:Hilb.welldefn.M4}, and for all $\beta\in\OO$ divisible by $2$, by the above claim.
Since the Hilbert symbol is non-degenerate on $K_2/K_2^\times$ \cite[Chapter XIV, Proposition 7]{SerreLF}, this implies that $\alpha+4B\in\OO^2$. Hence $[\alpha]=[\alpha+4B]$ is trivial in $\mathbf{M}_4$.
\end{proof}


For $\mmm$ an ideal of $K$, let $\PP_K^\mmm$ denote the set of prime ideals of $\OO$ co-prime to $\mmm$. 
For $K$ a totally real number field satisfying (C3),
we can define the following map.

\begin{defn}\label{defn:r} \switchAprint{[defn:r]}
    For $q$ a power of $2$, define the map
\begin{align*}
\mathbf{r}_q: & 
\PP_K^{2}\to  \mathbf{M}_q \\
 & \pp \mapsto \alpha
\end{align*}
where $\alpha\in\OO$ is a totally positive generator of the principal ideal $\pp^{\text{$h$}}$. 
\end{defn}


By Lemma \ref{RCFh}\switchAprint{RCFh} since $K$ is totally real with odd narrow class number, all principal ideals have a totally positive generator and $\OO^{\times}_{+} = \left(\OO^{\times}\right)^2$. Since squares are trivial in $\mathbf{M}_4$ by definition 
the map $\mathbf{r}_q$ is well-defined.
Note that $\mathbf{r}_q$ commutes with the Galois action, i.e. $\mathbf{r}_q(\pp^\sigma) = \mathbf{r}_q(\pp)^\sigma$ for all $\sigma\in\Gal(K/\QQ)$.

For $\mmm$ an ideal of $\OO$, let $J_K^\mmm$ denote the group of fractional ideals of $K$ prime to $\mmm$. 

\begin{lem}[{\cite[Lemma~3.5]{APSR}}]\label{lem:homMK}\switchAprint{[lem:homMK]}
For $K$ totally real with $h^+$ odd, the homomorphism $J_K^2 \to \mathbf{M}_4$ induced by $\mathbf{r}_q$ induces a 
surjective homomorphism, 
\[
\varphi_q: \CL^+_{q} \to  \mathbf{M}_q.
\]
\end{lem}

\begin{proof}

We expand on the proof of \cite[3.5]{APSR}. There the result is stated with more assumptions, but the same proof holds more generally. 

The proof from \cite[3.5]{APSR} shows that $\varphi_q$ is well-defined. We elaborate on the proof of surjectivity provided there.
Let 
\begin{align*}
&K_\mm:=\{a\in K^\times: \ord_2(a)=0\}, \quad \text{and}\\
&K_{\mm,1}:=\{a \in K^\times: \ord_2(a-1)\geq \ord_2(q), \totpos{a}\}
\end{align*}
 
 We have the following commutative diagram of homomorphisms. The homomorphism $\psi_0$ and the isomorphism $i$ are induced by the exact sequence and canonical isomorphism from class field theory as given in \cite[V.1.7]{MilneCFT}.

\begin{center}
\begin{tikzpicture}
\diagram (m)
{ (K_\mm/K_{\mm,1})/(K_\mm/K_{\mm,1})^2 & \CL_{q}^+/(\CL_{q}^+)^2 & \mathbf{M}_q  \\
 (\pm)^n \times  \mathbf{M}_q & & \\};
\path [->] 
    (m-1-1) edge node [above] {$\psi_0$} (m-1-2)
    (m-1-2) edge node [above] {$\varphi_q$} (m-1-3)
    (m-1-1) edge node [left] {$i$} (m-2-1)
    (m-2-1) edge node [below] {$\psi$} (m-1-2)
    ;
\end{tikzpicture}
\end{center}

Fix $X\in \mathbf{M}_q$. Consider $(1,X)\in (\pm)^n\times \mathbf{M}_q$. Since $i$ is an isomorphism, there exists and element in $(K_\mm/K_{\mm,1})/(K_\mm/K_{\mm,1})^2$ represented by $\beta\in K^\times$ such that $i(\beta)=(1,X)$. Since $i(\beta)$ maps to $1$ in the projection to $(\pm)^n$, $\beta$ is totally positive. Since $\totpos{\beta}$, we can  choose $a,b\in \OO$ totally positive such that $\beta=a/b$. (Writing $\beta=a/b$ for any $a,b\in \OO$, one could then consider $\beta=a^2/ab$). Then $X=[ab\inv]$ by the canonical isomorphism in \cite[V.1.7]{MilneCFT}.

The map $\psi_0$ takes $\beta$ to the class represented by the fractional ideal $(a)(b)\inv$. Since $a$ and $b$ are totally positive, $\varphi_q((a)(b)\inv)=[ab\inv]=X$ and so $\varphi_q$ is surjective.
\end{proof}

\begin{lem}[{\cite[Lemma~4.3]{APSR}}]\label{equidistribution}\switchAprint{[equidistribution]} Assume $K$ satisfies (P1), (C3), and (C4).

\begin{enumerate}
    \item\label{equidistribution:allp}  For any $\alpha\in \mathbf{M}_4$, the density of primes $\pp$ of $K$ such that $\varphi_4(\pp)=\alpha$ is $\frac{1}{2^n}$. That is, 
    \[
    d(\mathbf{r}_4\inv(\alpha)) = \frac{1}{\#\mathbf{M}_4} = \frac{1}{2^n}.
    \]
    \item\label{equidistribution:spcomp}  Furthermore, the density does not change when we restrict to primes of $K$ that split completely in $K/\QQ$. That is, 
    \[
    d(\mathbf{r}_4\inv(\alpha) \cap S' | S') = \frac{1}{\#\mathbf{M}_4} = \frac{1}{2^n}.
    \]
\end{enumerate}
\end{lem}

\begin{proof}
See Lemma 4.3 in \cite{APSR}. There the result is stated with more assumptions, but the same proof holds more generally. 
\end{proof}

\begin{defn}\label{defn:NMG}\switchAprint{defn:NMG}
Assume $K$ satisfies (P1), (C3), and (C4).
Let $\alpha\in\MK$. Let $\pp$ be an odd prime of $K$ such that $\mathbf{r}_4(\pp)=\alpha$. The map 
\begin{align*}
\mathbf{N}: \MK &\to (\ZZ/4)^\times \\
  \alpha &\mapsto \Norm_{K/\QQ}(\pp) \bmod 4\ZZ
\end{align*}
is well-defined and $\mathbf{N}(\alpha)=\mathbf{N}(\alpha^\sigma)$ for all $\sigma\in\Gal(K/\QQ)$. 
\end{defn}

\begin{proof}
 By  Lemma \ref{equidistribution}\switchAprint{equidistribution}, for any $\alpha\in \MK$, there exists a prime $\pp\in\PP_K^{2}$ such that $\mathbf{r}_4(\pp)=\alpha$. 
 
 Let $\pp$ and $\qq$ be (odd) primes of $K$ such that $\mathbf{r}_4(\pp)=\mathbf{r}_4(\qq)$. Let $\alpha$ be a totally positive generator of $\pp^h$ and let $\beta$ be a totally positive generator of $\qq^h$, where $h$ is the (odd) class number of $K$. Since $\mathbf{r}_4(\pp)=\mathbf{r}_4(\qq)$, $\alpha \equiv \beta$ in $\mathbf{M}_4$. Then $\alpha \equiv \beta \gamma^2\bmod 4\OO$ for some $\gamma\in \OO$. Since $2$ is inert, $\alpha^\sigma \equiv \beta^\sigma (\gamma^\sigma)^2\bmod 4\OO$  for all $\sigma\in\Gal(K/\QQ)$. Therefore $\Norm(\alpha) \equiv \Norm(\beta) \Norm(\gamma)^2 \bmod 4\OO$. Since the norms are in $\ZZ$, $\Norm(\alpha) \equiv \Norm(\beta)\bmod 4\ZZ$. 
 \end{proof}

We now state an extended version of Lemma \ref{equidistribution}\switchAprint{equidistribution} that handles the densities restricted to primes of a fixed congruence class modulo $4\ZZ$. 

\begin{lem}\label{equidistribution:extpm}\switchAprint{[equidistribution:extpm]} 
Assume $K$ satisfies conditions (C1)-(C4).
    For a fixed sign $\pm$, let $S_{\pm}'$ denote the set of primes of $K$ laying above some $p\in S$ such that $p\equiv \pm 1 \bmod 4\ZZ$.
    For any $\alpha\in \mathbf{M}_4$, the density of $\pp\in S_{\pm}'$ such that $\varphi_4(\pp)=\alpha$ is given by
    \[
    d(\mathbf{r}_4\inv(\alpha) \cap S_{\pm}'|S_{\pm}') = \left\{
    \begin{array}{ll}
        \frac{1}{2^{n-1}} &  \text{if } \mathbf{N}(\alpha) = \pm 1 \bmod 4\\
        0 & \text{otherwise.}
    \end{array}
    \right.
    \]
\end{lem}

\begin{proof}
By Lemma \ref{lem:homMK}\switchAprint{lem:homMK}, the map $\mathbf{r}_4: \PP_K^{2} \to \MK$ from Definition \ref{defn:r}\switchAprint{defn:r} induces a surjective 
group homomorphism
\[
\varphi_4: \CL_4^+ \surj \MK.
\]
and this homomorphism commutes with the action from $\Gal(K/\QQ)$. 
Define $H\defeq  \Art(\ker(\varphi_4))$ where $\Art$ denotes the Artin map.

Let $R_4^+$ denote the narrow ray class field over $K$ of conductor $4$.
Define $L$ to be the fixed field of $H$ 
in $\Gal(R_4^+/K)$. Then $\varphi_4$ induces a canonical isomorphism 
\[
\Gal(L/K)\ism \MK,
\]
which commutes with the action from $\Gal(K/\QQ)$.

Let $K_{4,1}\defeq\{\alpha\in K^\times: \ord_2(\alpha-1)\geq 2, \alpha\succ 0\}$ and $\OO^\times_{4,1} \defeq K_{4,1} \cap \OO^\times$. Since $\OO^\times_+=(\OO^\times)^2$,
    \[
    (\OO^\times:\OO^\times_{4,1}) = (\OO^\times:\OO^\times_+)(\OO^\times_+:\OO^\times_{4,1}) = 2^n((\OO^\times)^2:\OO^\times_{4,1}).
    \]
Therefore by  \cite[V.1.7]{MilneCFT}, the order of $\Gal(R_4^+/K)$ divides $h2^n(2^n-1)$.

We know $[L:K]=2^n$ since $\Gal(L/K)\ism \mathbf{M}_4$ and $\#\MK=2^n$ by Proposition \ref{prop:Mundy:ext}\switchAprint{prop:Mundy:ext}. Therefore $[R_4^+:L]$ is odd. Let $F$ denote the composite of $K$ and $\QQ(\zeta_4)$. Since $[K:\QQ]$ is odd, $[F:K]=2$. Since $F \subseteq R_4^+$ and $[R_4^+:L]$ is odd, $F\subseteq L$ and $[L:F]=2^{n-1}$. 


For $T/E$ a Galois extension of conductor dividing $\mmm$, let $p$ be a prime of $E$, and let $\tau\in\Gal(T/E)$. Let $(p,T/E)$ denote the conjugacy class of $\Gal(T/E)$ containing the Frobenius of $\pp$ where $\pp$ is a prime of $T$ above $p$. Let
\begin{align*}
\preArt{T|E}{E}(\tau)&\defeq \{ p\in\PP_E^{\mmm}: (p,T/E) = \left<\tau\right>\},\\
\preArt{T|E}{T}(\tau)&\defeq \{ \pp\in\PP_T^{\mmm}: \pp \text{ lies above }p\in\preArt{T|E}{E}(\tau) \}.
\end{align*}

Let $\alpha\in\MK$ and let $\sigma\in\Gal(L/K)$ corresponding to $\alpha$ via the isomorphism induced by $\varphi_4$. 

Fix a sign $\pm$ and let $\tau_0\in\Gal(\QQ(\zeta_4)/\QQ)$ such that 
\[
\preArt{\QQ(\zeta_4)/\QQ}{\QQ}(\tau_0) = \{ p \in \PP_\QQ^{2}: p \equiv \pm 1 \bmod 4\ZZ\}.
\]
Since $n=[K:\QQ]$ is odd,
$\Gal(F/K) \ism \Gal(\QQ(\zeta_4)/\QQ)$. Let $\tau\in\Gal(F/K)$ corresponding to $\tau_0$ in $\Gal(\QQ(\zeta_4)/\QQ)$.

Recalling that $\varphi_4$ is induced by $\mathbf{r}_4$, observe that 
\[
    \mathbf{r}_4\inv(\alpha) = \preArt{L/K}{K}(\sigma), \quad 
    S' = \preArt{K/\QQ}{K}(1), \quad \text{and} \quad
    S_\pm' = \preArt{F/K}{K}(\tau) \cap S'.
\]
Then the density in question is
\[
d_\pm\defeq  d\left( \frac{\mathbf{r}_4\inv(\alpha) \cap S_\pm' }{S_\pm'} \right)
= d \left( \frac{\preArt{K/\QQ}{K}(1) \cap \preArt{L/K}{K}(\sigma) \cap \preArt{F/K}{K}(\tau)}{\preArt{K/\QQ}{K}(1) \cap \preArt{F/K}{K}(\tau)} \right)
\]

Consider $\bar{\sigma}\in \Gal(F/K)$ taken to be the image of $\sigma$ under the natural surjection,
\[
\Gal(L/K) \twoheadrightarrow \Gal(F/K).
\]
Note that $\bar{\sigma} = \tau$ exactly when $\mathbf{N}(\alpha) = \pm 1 \bmod 4$. 
If $\bar{\sigma} \neq \tau$ then 
$
\mathbf{r}_4\inv(\alpha) \cap S_\pm' = \emptyset
$
so the density in question is 0. We now assume $\mathbf{N}(\alpha) = \pm 1 \bmod 4$ so that $\bar{\sigma} = \tau$. Then 
\[
\preArt{L/K}{K}(\sigma) \cap \preArt{F/K}{K}(\tau) = \preArt{L/K}{K}(\sigma).
\]
Therefore
\[
d_\pm = d \left( \frac{\preArt{K/\QQ}{K}(1) \cap \preArt{L/K}{K}(\sigma) }{\preArt{K/\QQ}{K}(1) \cap \preArt{F/K}{K}(\bar{\sigma})} \right).
\]
Restricting to primes of norm over $\QQ$ less than $N$, there are surjective maps of the following indices
\[
\preArt{K/\QQ}{K}(1) \cap \preArt{L/K}{K}(\sigma) |_N \rightarrow \preArt{L/\QQ}{\QQ}(\sigma)|_N \quad \text{ with index = } \#\Stab(\sigma)
\]
and
\[
\preArt{K/\QQ}{K}(1) \cap \preArt{F/K}{K}(\bar{\sigma}) |_N \rightarrow \preArt{F/\QQ}{\QQ}(\bar{\sigma})|_N \quad \text{ with index = } \#\Stab(\bar{\sigma})
\]
where $\Stab(\sigma)$ and $\Stab(\bar{\sigma})$ denote the stabalizers of $\sigma$ and $\bar{\sigma}$ respectively under the action from $\Gal(K/\QQ)$. 


Then by the Chebotarev Density Theorem (see Theorem  \cite[VII.13.4]{NeukirchANT} for Dirichlet density or \cite{Chebotarev} for natural density) and by the Orbit-Stabilizer Theorem,
\begin{align*}
d_\pm &= \frac{\#\Stab(\sigma)}{\#\Stab(\bar{\sigma})}
d \left(\frac{
\preArt{L/\QQ}{\QQ}(\sigma)
}{
\preArt{F/\QQ}{\QQ}(\bar{\sigma})
}
\right) \\
& = \frac{\#\Stab(\sigma) d(\preArt{L/\QQ}{\QQ}(\sigma)) }{\#\Stab(\bar{\sigma}) d(\preArt{F/\QQ}{\QQ}(\bar{\sigma}))} 
\\
& = \left(\frac{\#\Stab(\sigma)\#\left<\sigma\right>}{\#\Stab(\bar{\sigma})\#\Gal(L/\QQ)}\right)
\bigg/
 \left(\frac{\#\left<\bar{\sigma}\right>}{\#\Gal(F/\QQ)}\right)
 \\
& = \frac{\#\Stab(\sigma)\#\left<\sigma\right>}{\#\Stab(\bar{\sigma})\#\left<\bar{\sigma}\right>\#\Gal(L/F)}  \\
& = \frac{1}{\#\Gal(L/F)} \\
&= \frac{1}{2^{n-1}}. \qedhere
\end{align*}
\end{proof}

 Recall that Proposition \ref{prop:spinrelation}\switchAprint{prop:spinrelation} 
 states that for $\pp$ a prime of $K$ with totally positive generator $\alpha\in\OO$, and for $\sigma\in\Gal(K/\QQ)$ such that $\pp$ and $\pp^\sigma$ are relatively prime,
\[
\spin(\pp,\sigma)\spin(\pp,\sigma\inv) = (\alpha,\alpha^\sigma)_2.
\]
This motivates the following definition.

\begin{defn}[{\cite[Theorem 5.1]{APSR} }]
\label{defn:StarMG}\switchAprint{defn:StarMG} 
Assume $K$ is Galois with abelian Galois group and satisfies (C4).
Let $\alpha\in\OO$ denote a
representative of $[\alpha]\in$ $\mathbf{M}_{4}$. 
Define the map
\begin{align*}
\text{$\star$}:  \text{ $\mathbf{M}_{4}$ }& \to \{\pm 1\} \\
  [\alpha] &\mapsto \left\{ 
  \begin{array}{l l}
    1 & \quad \text{if }  (\alpha,\alpha^\sigma)_2 = 1
    \text{ for all non-trivial } \sigma\in \Gal(K/\QQ),  \\
    -1 & \quad \text{otherwise.}\\
  \end{array} \right.
\end{align*}
\end{defn}
Observe that $\star$ is a well-defined map by Lemma~\ref{lem:Hilb.welldefn.M4}. 
If \eqref{spin} holds for some $\alpha\in\OO$, then it holds for $\alpha^{\sigma}$ for any $\sigma\in\Gal(K/\QQ)$.  Therefore $\star(\alpha)=\star(\alpha^\sigma)$ for all $\sigma\in\Gal(K/\QQ)$.
Recall the map $\mathbf{N}: \mathbf{M}_4 \to \pm 1$ from Definition \ref{defn:NMG}\switchAprint{defn:NMG}. 
Let $\star_+$ denote the restriction of $\star$ to 
\[\MK^+\defeq \{\alpha \in \MK: \mathbf{N}(\alpha)=1\}\]
and let $\star_-$ denote the restriction of $\star$ to 
\[\MK^-\defeq \{\alpha \in \MK: \mathbf{N}(\alpha)=-1\}.\]

Define $S_+$ to be the set of odd rational primes congruent to $1 \mod 4$ that split completely in $K/\QQ$. Similarly, define $S_-$ to be the set of odd rational primes congruent to $-1 \mod 4$ that split completely in $K/\QQ$. 
Recall
\[
R \defeq \{p\in S: \spin(\pp,\sigma) \spin(\pp,\sigma\inv)=1  \text{ for all } \sigma\neq 1 \in \Gal(K/\QQ)\}.
\]
For a fixed sign $\pm$, define $R_\pm \defeq R\cap S_\pm$.

\begin{thm}\label{thm:pmdensityformulas}\switchAprint{thm:pmdensityformulas}
Assume $K$ satisfies properties (C1)-(C4).
Then
\[
d(R|S) = \frac{\#\ker(\star)}{2^n},
\]
\[
d(R_+|S_+) = \frac{\#\ker(\star_+)}{2^{n-1}} \quad \text{and} \quad
d(R_-|S_-) = \frac{\#\ker(\star_-)}{2^{n-1}}.
\]
\end{thm}

\begin{proof} 

That $d(R|S) = \#\ker(\star)/2^n$ is proven in \cite[Theorem~6.2]{APSR}, though it will also follow from the proof that
$
d(R_\pm|S_\pm) = {\# \ker(\star_\pm)}/{ 2^{n-1}}
$
since $d(S_\pm|S)=1/2$ and $\ker(\star)$ is the disjoint union of $\ker(\star_+)$ and $\ker(\star_-)$ and $R$ is the disjoint union of $R_+$ and $R_-$.

Recall the map $\mathbf{r}_4$ from Definition \ref{defn:r}\switchAprint{defn:r}. As shown in Definition \ref{defn:StarMG}, $\star(\alpha)=\star(\alpha^\sigma)$ for any $\sigma\in\Gal(K/\QQ)$ so $\star\circ\mathbf{r}_4(\pp)=\star\circ\mathbf{r}_4(\pp^\sigma)$ for any $\sigma\in\Gal(K/\QQ)$.
 By Proposition \ref{prop:spinrelation}\switchAprint{prop:spinrelation}, for each fixed sign $\pm$,
 \[
 R_\pm = \{p\in S_\pm: \star\circ\mathbf{r}_4(\pp)=1 \text{ for }\pp \text{ a prime of $K$ above $p$}\}.
 \]
 
For $N\in\ZZ_+$, let $R_{\pm,N}\defeq \{p\in R_{\pm}: p<N\}$ and   $S_{\pm,N}\defeq \{p\in S_{\pm}: p<N\}$. 
We will prove that 
\[
d(R_\pm|S_\pm) = \frac{\# \ker(\star_\pm)}{ \# \MK^\pm}.
\]
Then since $K$ is cyclic of odd degree and 2 is inert in $K/\QQ$, we can apply Proposition \ref{prop:Mundy:ext}\switchAprint{prop:Mundy:ext} to get that $\#\MK = 2^n$. Then since half the elements of $\MK$ are in $\MK^+$ and half in $\MK^-$, $\#\MK^+=\#\MK^- = 2^{n-1}$. 

Let $\pm$ denote a fixed sign and let $\mp$ denote the opposite sign. Let $S_{\pm,N}'$ denote the set of primes of $K$ laying above primes in $S_{\pm,N}$ and let $R_{\pm,N}'$ denote the set of primes of $K$ laying above primes in $R_{\pm,N}$.
Since primes in $S$ split completely, 
\[
\frac{\#R_{\pm,N}}{\#S_{\pm,N}}= \frac{\#R_{\pm,N}'}{\#S_{\pm,N}'}.
\]
Let $\mathbf{r}_{4,N}$ denote the restriction of $\mathbf{r}_4$ to $S_{\pm, N}'$. Then $R_{\pm,N}'$ is the disjoint union
\[
R_{\pm,N}' = 
\bigsqcup_{\alpha\in\ker(\star_\pm)} \left(S_{\pm, N}' \cap \mathbf{r}_{4,N}\inv(\alpha)\right),
\]
taken over elements $\alpha \in \ker(\star_\pm)$, i.e. elements of $\alpha\in\MK$ such that $\mathbf{N}(\alpha)=\pm 1 \bmod 4$ and $\star(\alpha)=1$.
Therefore
\[
\frac{\#R_{\pm,N}'}{\#S_{\pm,N}'} = \sum_{\alpha\in\ker(\star_\pm)} \frac{\# \left(S_{\pm, N}' \cap \mathbf{r}_{4,N}\inv(\alpha)\right)}{\#S_{\pm,N}'}
\]
By Lemma \ref{equidistribution:extpm}\switchAprint{equidistribution:extpm}, for all $\alpha\in \ker(\star_\pm)$,
\[
d(\mathbf{r}_4\inv(\alpha)\cap S_{\pm}' | S_{\pm}') = \frac{1}{\#\MK^\pm} = \frac{1}{2^{n-1}}.
\]
Therefore
\begin{align*}
d(R_\pm|S_\pm) = d(R_\pm'|S_\pm') 
&= \lim_{N\to\infty} \sum_{\alpha\in\ker(\star_\pm)} \frac{\# \left(S_{\pm, N}' \cap \mathbf{r}_{4,N}\inv(\alpha)\right)}{\#S_{\pm,N}'}\\
&=\sum_{\alpha\in\ker(\star_\pm)} \lim_{N\to\infty} \frac{\# \left(S_{\pm, N}' \cap \mathbf{r}_{4,N}\inv(\alpha)\right)}{\#S_{\pm,N}'}\\
& =  \sum_{\alpha\in\ker(\star_\pm)} d(\mathbf{r}_4\inv(\alpha)\cap S_{\pm}' | S_{\pm}')\\
&= \sum_{\alpha\in\ker(\star_\pm)} \frac{1}{2^{n-1}} =\frac{\#\ker(\star_\pm) }{2^{n-1}}. \qedhere
\end{align*}
\end{proof}

\section{Counting Solutions to a Hilbert Symbol Condition}\label{sec:CountingHilbert}





In this section, we assume that $K/\QQ$ satisfies (C1)-(C4) and  prove formulae for $\#\ker(\star_{\pm})$.

Fix $\tau$ to be a generator of $\Gal(K/\QQ)$.
For any $\alpha\in K$, write $\alpha_{(k)}\coloneqq \alpha^{\tau^k}$ for $k\in\ZZ$.

\begin{lem}\label{minusone}
$(-1,-1)_2=-1$.
\end{lem}
\begin{proof}
Assume for contradiction that $(-1,1)_2=1$.
Consider a homomorphism $\psi:\mathbf{M}_4\rightarrow\{\pm 1\}$ given by $[\alpha]\mapsto(\alpha,-1)_2$.
Since the Hilbert symbol is non-degenerate, and $-1$ is not a square modulo $4$ in $K$, $\psi$ is not identically $1$. Therefore $\#\ker\psi=\#\mathbf{M}_4/\#\im\psi=2^{n-1}$.

For any $[\alpha]\in\mathbf{M}_4\setminus\{\pm 1\}$, we have $(\alpha_{(k)},-1)_2=(\alpha,-1)_2$ for any $k$.
Therefore $\psi$ is stable under the Galois action. The size of each Galois orbit is $n$ except the orbit of $\pm 1$.
But then $n$ divides both $\#\{[\alpha]\in\mathbf{M}_4\setminus\{\pm 1\} :\psi(\alpha)=1\}=\#\{[\alpha]\in\mathbf{M}_4 :\psi(\alpha)=1\}-2=2^{n-1}-2$ and $\#\{[\alpha]\in\mathbf{M}_4 :\psi(\alpha)=-1\}=2^{n-1}$, which is a contradiction.
\end{proof}
Our aim is to count the number of elements in $\mathbf{M}_4$ with a representative $\alpha\in \OO_K$ satisfying the spin relation
\begin{equation}\label{spin}
(\alpha,\alpha^\sigma)_2=1 \text{ for all non-trivial }\sigma\in\Gal(K/\QQ).
\end{equation}
By Lemma~\ref{lem:Hilb.nondeg.M4}, the property~\eqref{spin} only depends on the class of $[\alpha]\in\mathbf{M}_4$.


\subsection{The Hilbert symbol as a bilinear form on \texorpdfstring{$\mathbf{M}_4$}{M4}}

By the Kronecker--Weber theorem, $K$ is contained in the cyclotomic field $\QQ(\zeta_{\mathfrak{f}})$, where $\mathfrak{f}$ is the conductor of $K$. The conductor $\mathfrak{f}$ is odd since we assumed that $2$ is unramified in $K$.
By~\cite[Theorem 4.5]{FGS}, there exists a normal $2$-integral basis of $\QQ(\zeta_{\mathfrak{f}})$, i.e.\ we can find some $a\in\OO_{\QQ(\zeta_{\mathfrak{f}})}$ such that the  localization of $\OO_{\QQ(\zeta_{\mathfrak{f}})}$ at $2$ can be written as $\OO_{\QQ(\zeta_{\mathfrak{f}}),2}=\oplus_{g\in \Gal(\QQ(\zeta_{\mathfrak{f}})/\QQ)}\ZZ_{(2)}a^{g}$.
Similar to the classic result for integral bases~\cite[Proposition~4.31(i)]{Narkiewicz}, taking $y=\tr_{\QQ(\zeta_{\mathfrak{f}})/K} (a)$, then $\{y,y^\tau, \dots, y^{\tau^{n-1}}\}$ gives a normal $2$-integral basis of $K$.
Since $\ZZ_{(2)}/2\cong \ZZ/2$ and $\OO_{K,2}/2\cong \OO_K/2$, we know that $y,y^\tau, \dots, y^{\tau^{n-1}}$ also form a normal $\FF_2$-basis of $\OO_K/2$.

Set $\alpha=1+2y$. It follows from the isomorphism in~\eqref{eq:M4iso} that
\[\mathbf{M}_4=\left\{\prod_{i=0}^{n-1}[\alpha_{(i)}]^{u_i}:(u_0,\dots,u_{n-1})\in\FF_2^n\right\}.\]
Write $(\alpha,\alpha_{(i)})_2=(-1)^{c_i}$, $c_i\in\{0,1\}$.
Note that $(\alpha_{(i)},\alpha_{(j)})_2=(\alpha,\alpha_{(j-i)})_2$.
The Hilbert symbol is multiplicatively bilinear, so we can represent $(\parm,\parm)_2$ by the matrix
\begin{equation}\label{eq:defA}
A\coloneqq 
\begin{pmatrix}
  c_0     & c_{n-1} & c_{n-2} & \dots & c_1 \\
  c_1     & c_0     & c_{n-1} & \dots & c_2 \\
  c_2     & c_1     & c_0     & \dots & c_3 \\
  \vdots  & \vdots  & \vdots  & \ddots& \vdots\\
  c_{n-1} & c_{n-2} & c_{n-3} & \dots & c_0 \\
\end{pmatrix}
\end{equation}
with respect to the basis $[\alpha_{(i)}]$, $0\leq i\leq n-1$.

Define the $n\times n$ $\FF_2$-matrix 
\[T_1=
\begin{pmatrix}
  0 & 1 & 0 & 0 & \dots & 0 \\
  0 & 0 & 1 & 0 & \dots & 0 \\
  0 & 0 & 0 & 1 & \dots & 0 \\
  \vdots & \vdots & \vdots & \vdots & \ddots & \vdots\\ 
  0 & 0 & 0 & 0 & \dots & 1 \\
  1 & 0 & 0 & 0 & \dots & 0 \\
\end{pmatrix},
\]
$T_k=T_1^k$ and $T_0=I$.

\begin{lem}
Let $A$ be the matrix representation of $(\parm,\parm)_2$ on $\mathbf{M}_4$ with respect to a normal basis, as given in \eqref{eq:defA}.
Define a map
\begin{align*}
\Psi:\ &\FF_2^n\rightarrow \FF_2[x]/(x^n-1)\\
&\mathbf{u}=(u_0,\dots,u_{n-1})
\mapsto F_{\mathbf{u}}(x)\coloneqq u_0+u_1x+u_2x^2+\dots+u_{n-1}x^{n-1}.
\end{align*}
Also define
\[
\Phi: 
\FF_2^n
\rightarrow  \FF_2^n
\qquad \mathbf{u}
\mapsto 
(
  \mathbf{u}^T T_0\mathbf{u},\
  \mathbf{u}^T T_1\mathbf{u},\ 
  \dots,\
  \mathbf{u}^T T_{n-1}\mathbf{u}
).
\]
Let $B\coloneqq\Psi\circ\Phi$, so
\[
B:
\FF_2^n
\rightarrow \FF_2[x]/(x^n-1)
\qquad
\mathbf{u}
\mapsto x^n\cdot F_{\mathbf{u}}(x)F_{\mathbf{u}}(1/x)\bmod (x^n-1).
\]
Then
$\#\ker(\star_+) 
=\#B^{-1}(0)$ and $\#\ker(\star_-) =\# B^{-1}(h(x))$, where $h(x)=\Psi(A^{-1}(1,0,\dots,0))$. Furthermore 
\begin{equation}\label{eq:hreal}
h(x)\equiv x^n h(1/x)\bmod (x^n-1).
\end{equation}
\end{lem}
\begin{proof}
For any $\mathbf{u}=(u_0,\dots,u_{n-1}),
\mathbf{v}=(v_0,\dots,v_{n-1})\in\FF_2^{n}$,
we have
\[\left(\prod_{i}\alpha_{(i)}^{u_i},\prod_{j}\alpha_{(j)}^{v_j}\right)_2
=(-1)^{\mathbf{u}^TA\mathbf{v}}.\]
Since $(\parm,\parm)_2$ is non-degenerate on $\mathbf{M}_4$ by Lemma~\ref{lem:Hilb.nondeg.M4}, the matrix $A$ has rank $n$ and is invertible. Note also that $A$ is symmetric.

Now $\prod_{i}\alpha_{(i)}^{u_i}$, $\mathbf{u}=(u_0,\dots,u_{n-1})\in\FF_2^{n}$  satisfies~\eqref{spin} if and only if 
\begin{equation}\label{matrixcond}
\mathbf{u}^TAT_1\mathbf{u}=\mathbf{u}^TAT_2\mathbf{u}=\dots=\mathbf{u}^TAT_{n-1}\mathbf{u}=0.
\end{equation}
Since $\{T_0,T_1,\dots,T_{n-1}\}$ is a basis of $\GL_n(\FF_2)$, we can write 
\[A=\sum_{i=0}^{n-1}c_iT_i,\qquad c_i\in \FF_2.\]
Then~\eqref{matrixcond} becomes
\begin{equation}\label{eq:Acond}
A\circ\Phi(\mathbf{u})=A
\begin{pmatrix}
  \mathbf{u}^T T_0\mathbf{u} \\
  \mathbf{u}^T T_1\mathbf{u} \\
  \vdots\\ 
  \mathbf{u}^T T_{n-1}\mathbf{u} \\
\end{pmatrix}
\in 
\left\{
\begin{pmatrix}
  0 \\
  0 \\
  \vdots\\ 
  0 \\
\end{pmatrix},
\begin{pmatrix}
  1 \\
  0 \\
  \vdots\\ 
  0 \\
\end{pmatrix}
\right\}.
\end{equation}

Since $A$ is invertible, we can set $h(x)=\Psi(A^{-1}(1,0,\dots,0))$.  
Notice that $\Psi$ is a one-to-one correspondence. Then~\eqref{eq:Acond} can be rewritten as $B(\mathbf{u})=\Psi\circ\Phi(\mathbf{u})\in\{0,h(x)\}$. 
Since $A$ is symmetric, $A^{-1}$ is also symmetric, so \eqref{eq:hreal} holds.
Also $(\alpha,\alpha)_2=(\alpha,-1)_2=(\alpha,-1)_2^n=\prod_{i}(\alpha_{(i)},-1)_2=(\Norm_{K/\QQ}(\alpha),-1)_2$, which is $1$ if $\Norm_{K/\QQ}(\alpha)\equiv 1\bmod 4$ and $-1$ if $\Norm_{K/\QQ}(\alpha)\equiv -1\bmod 4$ by Lemma~\ref{minusone}. Therefore
$\#\ker(\star_+) 
=\#B^{-1}(0)$ and $\#\ker(\star_-) =\# B^{-1}(h(x))$.
\end{proof}
\subsection{The counting problem}
Our aim is to obtain the size of the preimage of $0$ and $h(x)$ under $B$.
For any polynomial $f$, let $f^*$ denote its reciprocal, i.e.\ $f^*(x)=x^{\deg f}\cdot f(1/x)$.

\begin{lem}\label{lemma:factor}
For any factor $k\neq 1$ of $n$, let $d_k$ be the order of $2$ in $(\ZZ/k)^\times$. Also set $d_1=1$.
Consider the following factorisation in $\FF_2[x]$, 
\[
x^n-1= f_1(x)\dots f_r(x)f^*_{m+1}(x)\dots f^*_r(x),
\]
where $f_i$ are irreducible and $f_i=f_i^*$ for $i=1,\dots, m$.
 Then
$\sum_{i=1}^r\deg f_i=\sum_{k\mid n}r_kd_k$ and $r=\sum_{k\mid n}r_k$ and $m=\sum_{k\mid n}m_k$, where $r_1=m_1=1$, and 
\[(r_k,m_k)=
\begin{cases}
\hfil \left(\frac{\phi(k)}{2d_k},\ 0\right)& \text{ if }d_k\text{ is odd,}\\
\left(\frac{\phi(k)}{d_k},\ \frac{\phi(k)}{d_k}\right)& \text{ if }d_k\text{ is even,}
\end{cases}\]
for $k\neq 1$.
\end{lem}

\begin{proof}
Take $f$ to be an irreducible factor of $x^n-1$ in $\FF_2[x]$.
Let $\gamma$ be a root of $f$ in an extension of $\FF_2$. Then $\gamma$ is a primitive $k$-th root of unity, where $k$ is some integer dividing $n$. 
Galois theory on finite fields shows that $\Gal(\FF_2(\gamma)/\FF_2)$ is generated by the Frobenius $\varphi: x\mapsto x^2$. Since $\varphi^i: x\mapsto x^{2^i}$ for any $i\in\ZZ$, we see that the order of $\varphi$ must be $d_k$, the order of $2$ in $(\ZZ/k)^{\times}$. Therefore $\deg f=d_k$. 
The set of roots of $f$ is $\{\gamma, \varphi(\gamma),\varphi^2(\gamma),\dots, \varphi^{d_k-1}(\gamma)\}$, which is closed under inversion precisely when $d_k$ is even. Therefore $f$ is self-reciprocal if and only if $d_k$ is even.
There are $\phi(k)$ roots of $x^n-1$ which are primitive $k$-th root of unity, so $(2r_k-m_k)d_k=\phi(k)$. 
\end{proof}
We are now ready to prove the formulae for $\#\ker(\star_+)$ and $\#\ker(\star_-)$.

\begin{prop}\label{prop:Starlight}
For each $k\neq 1$ dividing $n$, let $d_k$ be the order of $2$ in $(\ZZ/k)^\times$. Then
\[\#\ker(\star_+)=1+\prod_{k\mid n,\ d_k\text{odd},\ k\neq 1}2^{\frac{\phi(k)}{2d_k}}\left(\prod_{k\mid n,\ d_k\text{odd},\ k\neq 1}2^{\frac{\phi(k)}{2}}-1\right),\]
and
\[\#\ker(\star_-)=\prod_{k\mid n,\ d_k\text{even},\ k\neq 1}(2^{d_k/2}+1)^{\frac{\phi(k)}{d_k}}\prod_{k\mid n,\ d_k\text{odd},\ k\neq 1}(2^{d_k}-1)^{\frac{\phi(k)}{2d_k}},\]
where $\phi$ denotes the Euler's totient function.
If $n$ is a prime, then writing $d=d_n$, 
\[
(\#\ker(\star_+),\#\ker(\star_-))
=\begin{cases}
\left(1+2^{\frac{n-1}{2d}}(2^{\frac{n-1}{2}}-1),\ (2^d-1)^{\frac{n-1}{2d}}\right)
& \text{ if }d\text{ is odd,}\\
\hfil \left(1,\ (2^{\frac{d}{2}}+1)^{\frac{n-1}{d}}\right) & \text{ if }d\text{ is even.}
\end{cases}
\] 
In particular, when $n=3$, $\#\ker(\star_+)=1$ and $\#\ker(\star_-) =3$.
\end{prop}

\begin{proof}

The first case $B(\mathbf{u})=0$ implies
$(x^n-1)\mid F_{\mathbf{u}}(x)F^*_{\mathbf{u}}(x)$.
Obtain the following factorisation in $\FF_2[x]$ as described in Lemma~\ref{lemma:factor},
\begin{equation}\label{eq:factor}
x^n-1= f_1(x)\dots f_r(x)f^*_{m+1}(x)\dots f^*_r(x),
\end{equation}
where $f_1(x)=x+1$, $f_i$ are irreducible and $f_i=f_i^*$ for $i=1,\dots, m$.
Write $F_{\mathbf{u}}=G\cdot H$, where $G=\gcd(F_{\mathbf{u}},x^n-1)$.
Then for each $k=0,\dots, r$, we have $f_k\mid F_{\mathbf{u}}$ or $f_i^*\mid F_{\mathbf{u}}$.
This leaves us with $2^{r-m}$ choices for $G$. 
Since 
\[\deg\left((x^n-1)/G\right)=n-\sum_{i=1}^r\deg f_i=\sum_{k\mid n}(r_k-m_k)d_k,\]
There are $2^{\sum_{k\mid n}(r_k-m_k)d_k}-1$ choices for $H\not\equiv 0\bmod (x^n-1)/G$, so 
\begin{equation}\label{eq:B0count}
\#\ker(\star_+) 
=1+\#(B^{-1}(0)\setminus\{0\})
=1+2^{r-m}\left(2^{\sum_{k\mid n}(r_k-m_k)d_k}-1\right).
\end{equation}

The second case $B(\mathbf{u})=h(x)$. We count the number of $\mathbf{u}\in\FF_2^n$ such that
\begin{equation}\label{eq:B1}
x^n\cdot F_{\mathbf{u}}(x)F_{\mathbf{u}}(1/x)\equiv h(x)\bmod (x^n-1).
\end{equation}
Fix a primitive complex $n$-th root of unity $\zeta_n$. Consider the isomorphism
\begin{align*}
(\FF_2[x]/(x^n-1))^{\times}\rightarrow (\ZZ[\zeta_n]/2)^{\times}
&&
F_{\mathbf{u}}(x)
\mapsto 
F_{\mathbf{u}}(\zeta_n)\bmod 2.
\end{align*}
Now~\eqref{eq:B1} becomes
\[F_{\mathbf{u}}(\zeta_n)\overline{F_{\mathbf{u}}(\zeta_n)}\equiv h(\zeta_n)\bmod 2.\]
Notice from \eqref{eq:hreal} that $h(\zeta_n)=h(\zeta_n^{-1})=\overline{h(\zeta_n)}$ is real.
We compute from~\eqref{eq:factor}, 
\begin{multline*}
\#(\ZZ[\zeta_n]/2)^{\times}
=\#(\FF_2[x]/(x^n-1))^{\times}\\
=\prod_{i=1}^{r}\#\left(\FF_2[x]/(f_i)\right)^{\times}
\prod_{j={m+1}}^{r}\#\left(\FF_2[x]/(f^*_{j})\right)^{\times}
=\prod_{k\mid n}(2^{d_k}-1)^{2r_k-m_k}.\end{multline*}

Take $g\in\FF_2[x]$ such that
\[\frac{x^n-1}{x-1}\equiv x^{n-1}+x^{n-2}+\dots+x+1= x^{\frac{n-1}{2}}g(x+x^{-1}).\]
We can factorise
$g(x)=g_2(x)\dots g_r(x)$,
where $x^{\deg g_k}\cdot g_k(x+x^{-1})=f_k(x)$ for $2\leq k\leq m$ and 
$x^{\deg g_k}\cdot g_k(x+x^{-1})=f_k(x)f_k^*(x)$ for $m+1\leq k\leq r$.
Then since $(\ZZ[\zeta_n+\zeta_n^{-1}]/2)^{\times}\cong(\FF_2[x]/(g))^{\times}$, we compute
\begin{multline*}
\#\left(\ZZ[\zeta_n+\zeta_n^{-1}]/2\right)^{\times}
=\#(\FF_2[x]/(g))^{\times}\\
=\prod_{i=2}^{r}\#\left(\FF_2[x]/(g_i)\right)^{\times}
=\prod_{k\mid n,\ k\neq 1}(2^{d_k/2}-1)^{m_k}(2^{d_k}-1)^{r_k-m_k}.
\end{multline*}

Our goal is to compute the size of the kernel of the homomorphism
\begin{align*}
\psi:(\ZZ[\zeta_n]/2)^{\times}\rightarrow (\ZZ[\zeta_n+\zeta_n^{-1}]/2)^{\times}
&&
\beta\mapsto \beta\overline{\beta}.
\end{align*}
We claim that $\psi$ is surjective. Since $(\ZZ[\zeta_n+\zeta_n^{-1}]/2)^{\times}$ has odd order, every element is a square, so suppose $\beta^2\in(\ZZ[\zeta_n+\zeta_n^{-1}]/2)^{\times}$, then $\psi(\hat{\beta})=\beta^2$ for any lift $\hat{\beta}\in\ZZ[\zeta_n+\zeta_n^{-1}]$ of $\beta$.
Therefore
\begin{multline}\label{eq:B1count}
\#\ker(\star_-) =\# B^{-1}(h(x))
=\#\ker\psi\\
=\frac{\#(\ZZ[\zeta_n]/2)^{\times}}{\#\im\psi}
=\prod_{k\mid n,\ k\neq 1}(2^{d_k/2}+1)^{m_k}(2^{d_k}-1)^{r_k-m_k}.
\end{multline}

Putting in~\eqref{eq:B0count} and~\eqref{eq:B1count} the values of $r$ and $m$ in terms of $n$ and $d$ as in Lemma~\ref{lemma:factor} proves the proposition.
\end{proof}

\section{Joint Spins}\label{sec:dFR}
Fix a sign $\mu\in\{\pm\}$. Recall that $S_{\mu}$ is the set of rational primes $p\equiv \mu 1 \bmod 4$ that split completely in $K/\QQ$, i.e., unramified and of residue degree $1$ in $K/\QQ$, and that $F_{\mu}$ is the set of $p\in S_{\mu}$ of residue degree $1$ in $K(p)/\QQ$. By Corollary~\ref{cor:SpinRes}, a prime $p\in S_{\mu}$ belongs to $F_{\mu}$ if and only if $\spin(\pp, \sigma) = 1$ for all non-trivial $\sigma\in\Gal(K/\QQ)$ and any prime ideal $\pp$ of $K$ lying above $p$. Recall that $R_{\mu}$ is the set of primes $p\in S_{\mu}$ such that $\spin(\pp, \sigma)\spin(\pp, \sigma^{-1}) = 1$ for all non-trivial $\sigma\in\Gal(K/\QQ)$ and all prime ideals $\pp$ of $K$ lying above $p$, so that $F_{\mu}\subset R_{\mu}$. In this section, we will prove the following formula for the relative density of $F_{\mu}$ in $R_{\mu}$, denoted by $d(F_{\mu}|R_{\mu})$.
\begin{thm}\label{thm:JoinSpinCor}\switchAprint{thm:JoinSpinCor} 
Assume Conjecture~$C_{\eta}$ for $\eta = \frac{2}{n(n-1)}$. Then 
$$
d(F_{\mu}|R_{\mu}) =  2^{-\frac{n-1}{2}}.
$$
\end{thm}
Since each $p\in S_{\mu}$ splits into exactly the same number of prime ideals in $\OO$, and since $R_{\mu}$ is a set of primes of positive natural density, it suffices to show that  
\begin{equation}\label{dFR}
\sum_{\substack{\Norm(\pp)\leq X \\ \pp\text{ lies over }p\in F_{\mu}}}1 = 2^{-\frac{n-1}{2}}\sum_{\substack{\Norm(\pp)\leq X \\ \pp\text{ lies over }p\in R_{\mu}}}1 + o(X(\log X)^{-1}). 
\end{equation}
Let $\tau$ be a generator of $\Gal(K/\QQ)$, a cyclic group of order $n$. Then, by definition of the set $R_{\mu}$, a prime $p\in R_{\mu}$ belongs to the set $F_{\mu}$ if and only if $\spin(\pp,\tau^k)=1$ for all $k\in\{1, 2, \ldots, \frac{n-1}{2}\}$. The product
$$
\prod_{k = 1}^{\frac{n-1}{2}}\frac{1+\spin(\pp, \tau^k)}{2}
$$
is the indicator function of the property that $\spin(\pp, \tau^k) = 1$ for all $k\in\{1, 2,\ldots, \frac{n-1}{2}\}$. Expanding this product gives
\begin{equation}\label{indicator}
2^{-\frac{n-1}{2}}\sum_{H\subset \{\tau, \ldots, \tau^{\frac{n-1}{2}}\}}\prod_{\sigma\in H}\spin(\pp, \sigma),
\end{equation}
where the sum is over all subsets $H$ of $\{\tau, \tau^2, \ldots, \tau^{\frac{n-1}{2}}\}$. When $H = \emptyset$, the product is $1$ by convention.

Let $L/K$ be an abelian extension with Galois group isomorphic to $\MK^{\mu}$, and let $\mathcal{A}$ denote the set of disjoint $G$-orbits of elements of $\MK^{\mu}$, so that we can write
$$
\MK^{\mu} = \bigsqcup_{A\in\mathcal{A}} A.
$$
Each $G$-orbit $A$ is then a collection of invertible congruence classes modulo $4\OO$ that are distinct modulo squares. Let $\mathcal{A}_0\subset \mathcal{A}$ be the set of $G$-orbits $A$ such that $\spin(\pp, \sigma)\spin(\pp, \sigma^{-1})=1$ for all non-trivial $\sigma\in G$ and for all prime ideals $\pp$ such that $\mathbf{r}_4(\pp)\in A$. Note that a prime ideal $\pp$ in $\OO$ lies over a prime $p\in R_{\mu}$ if and only if $\mathbf{r}_4(\pp) \in A$ for some $A\in\mathcal{A}_0$.

Summing \eqref{indicator} over all prime ideals $\pp$ of norm $\Norm(\pp)\leq X$, we get that
$$
\sum_{\substack{\Norm(\pp)\leq X \\ p\in F_{\mu}}}1 = 2^{-\frac{n-1}{2}}\sum_{\substack{H\subset \{\tau, \ldots, \tau^{\frac{n-1}{2}}\} \\ A\in\mathcal{A}_0}}\Sigma(X; H, A),
$$
where
$$
\Sigma(X; H, A) = \sum_{\substack{\Norm(\pp)\leq X \\ \mathbf{r}_4(\pp)\in A}}\prod_{\sigma\in H}\spin(\pp, \sigma).
$$
The sums $\Sigma(X; \emptyset, A)$ feature no cancellation and provide the main term in \eqref{dFR}. 
It then remains to show that 
\begin{equation}\label{mainGoalJDS}
\Sigma(X; H, A) = o(X/\log X)
\end{equation}
for each non-empty subset $H$ of $\{\tau, \ldots, \tau^{\frac{n-1}{2}}\}$ and each $A\in\mathcal{A}_0$. To this end, we will use a slight generalization of Theorem~1 of \cite{JDS}.

We cannot apply the results of \cite{JDS} directly for two reasons. First, the class number $h$ of $K$ need not be $1$ -- this forces us to relate $\spin(\aaa, \sigma)$ to quadratic residue symbols involving elements ``smaller'' than the totally positive generators of $\aaa^h$. Second, the sums $\Sigma(X; H, A)$ feature the additional restriction that $\mathbf{r}_4(\pp)\in A$. Since $A$ is a collection of congruence classes modulo $4\OO$, the restriction that $\mathbf{r}_4(\pp)\in A$ is reminiscent of the restriction to a congruence class as in \cite[Theorem 1.2, p.\ 699]{FIMR}. Despite the similarity, there is a technical difference that we will explain. 

Fix once and for all a set $\mathcal{C}$ consisting of $h$ unramified degree-one prime ideals in $\OO$ that is a complete set of representatives of ideal classes in the class group of $K$; its existence is guaranteed by an application of the Chebotarev Density Theorem to the Hilbert class field of~$K$. 

Now suppose that $\aaa$ is a non-zero ideal in $\OO$ coprime to $\prod_{\pp\in\mathcal{C}}\Norm(\pp)$, and let $\alpha$ denote a totally positive generator of $\aaa^h$. As $h$ is odd, the set $\{\pp^2: \pp\in \mathcal{C}\}$ is also a complete set of representatives. Hence there exists $\pp\in\mathcal{C}$ such that $\aaa \pp^2$ is a principal ideal.  Let $\pi$ denote a totally positive generator of the ideal $\pp^h$. Let $\alpha_0$ denote a totally positive generator of $\aaa \pp^2$. Then $\alpha_0^h$ and $\alpha\pi^2$ are both totally positive generators of the ideal $(\aaa\pp^{2})^h$, so we have
\begin{equation}\label{eq:transform}
\spin(\aaa, \sigma) = \leg{\alpha}{\sigma(\aaa)} = \leg{\alpha\pi^2}{\sigma(\aaa\pp^2)} = \spin(\aaa\pp^2, \sigma) = \leg{\alpha_0^h}{\sigma(\aaa\pp^2)} =  \leg{\alpha_0}{\sigma(\alpha_0)},
\end{equation}
since $h$ is odd. Note that for each $\pp\in\mathcal{C}$ there is a bijection
\begin{multline}\label{eq:KeyBijection}
\{\aaa\subset \OO: \Norm(\aaa)\leq x, \aaa\pp^2\text{ is principal}\} \\ \simeq \{\alpha_0\in\DD: \Norm(\alpha_0)\leq x\Norm(\pp)^2, \alpha_0\equiv 0\bmod \pp^2\}
\end{multline}
given by $\aaa\mapsto \alpha_0$ as above, and where $\mathcal{D}$ is a set of totally positive elements in $\OO$ defined in~\cite[(4.2), p.713]{FIMR}. Moreover, $\mathbf{r}_4(\aaa)$ is the class in $\MK$ of a totally positive generator of $\aaa^h$, i.e., the class of $\alpha$ in $\MK$. Since squares vanish in $\MK$, the classes of $\alpha$ and $\alpha\pi^2$, and so also of $\alpha_0^h$, coincide in $\MK$. Hence, if $A$ is a $G$-orbit, then 
\begin{equation}\label{eq:r0Bijection}
\mathbf{r}_4(\aaa)\in A\quad\text{ if and only if }\quad \alpha_0^h\in A.
\end{equation}
We will now prove the following adaptation of \cite[Theorem~1, p.\ 2]{JDS}.

\begin{thm}\label{mainJointDist}
With notation as above, let $H$ be a non-empty subset of $\{\tau, \ldots, \tau^{\frac{n-1}{2}}\}$. Assume Conjecture~$C_{\eta}$ holds true for $\eta = 1/(|H|n)$ with $\delta = \delta(\eta) > 0$ (see \cite[p.\ 7]{JDS}). Let $\epsilon > 0$ be a real number.  Then for all $X\geq 2$, we have
$$
\Sigma(X; H, A) \ll X^{1 - \frac{\delta}{54|H|^2n(12n+1)} + \epsilon},
$$ 
where the implied constant depends only on $\epsilon$ and $K$.
\end{thm}
Note that the set $H$ above is of size at most $\frac{n-1}{2}$. Since Conjecture~$C_{\eta_1}$ implies Conjecture~$C_{\eta_2}$ whenever $\eta_1\leq \eta_2$, we see that, conditional on Conjecture~$C_{\eta}$ for $\eta = \frac{2}{n(n-1)}$, Theorem~\ref{mainJointDist} implies \eqref{mainGoalJDS} for each $G$-orbit $A\in\mathcal{A}_0$ and each non-empty subset $H\subset\{\tau, \ldots, \tau^{\frac{n-1}{2}}\}$, and hence also Theorem~\ref{thm:JoinSpinCor}. It thus remains to prove Theorem~\ref{mainJointDist}.

For a non-zero ideal $\aaa\subset \OO$ and a $G$-orbit $A$, let
$$
r(\aaa; A) = 
\begin{cases}
1 & \text{if }\mathbf{r}_4(\aaa)\in A \\
0 & \text{otherwise,}
\end{cases}
$$
and let
$$
s_{\aaa} = r(\aaa; A)\prod_{\sigma\in H}\spin(\aaa, \sigma).
$$
Then we have
$$
\Sigma(X; H, A) = \sum_{\Norm(\pp)\leq X} s_{\pp},
$$ 
where the summation is over prime ideals $\pp\subset \OO$ of norm at most $X$.

Let $F$ be the integer defined in \cite[(2.2), p.\ 5]{JDS}; it depends only on $K$. Moreover, we can choose the sets $\mathcal{C}\ell_a$ and $\mathcal{C}\ell_b$ in \cite[p.\ 5]{JDS} so that their elements are coprime to $\prod_{\pp\in\mathcal{C}}\Norm(\pp)$. Note that $F$ is divisible by $32$.

To deduce Theorem~\ref{mainJointDist}, it suffices to prove that
$$
\sum_{\substack{\Norm(\pp)\leq X \\ \pp\nmid F}} s_{\pp} \ll_{\epsilon, K} X^{1 - \frac{\delta}{54|H|^2n(12n+1)} + \epsilon}
$$
because $F$ has only finitely many prime ideal divisors. 

The proof of Theorem~\ref{mainJointDist} proceeds via Vinogradov's method, with suitable estimates necessary for the sums of type I
$$
A_{\mathfrak{m}}(x) = \sum_{\substack{\Norm{\mathfrak{a}} \leq x \\ (\mathfrak{a}, F) = 1,\ \mathfrak{m} \mid \mathfrak{a}}} s_{\aaa},
$$
where $\mathfrak{m}$ is any non-zero ideal coprime to $\tau(\mathfrak{m})$, and sums of type II
$$
B(x, y; v, w) = \sum_{\substack{\Norm(\aaa)\leq x\\ (\aaa, F) = 1}}\sum_{\substack{\Norm(\bb)\leq y\\ (\bb, F) = 1}}v_{\aaa}w_{\bb}s_{\aaa\bb},
$$
where $v = \{v_{\aaa}\}_{\aaa}$ and $w = \{w_{\bb}\}_{\bb}$ are arbitrary sequences of complex numbers of modulus bounded by $1$. By \cite[Proposition~5.2, p.\ 722]{FIMR} applied with $\vartheta = \frac{\delta}{54n|H|^2}$ and $\theta = \frac{1}{6n}$, the following two propositions imply Theorem~\ref{mainJointDist}.

\begin{prop}\label{prop:SumsTypeI}
Let $\delta = \delta(|H|n)>0$ be as in Conjecture~$C_{|H|n}$. Let $\epsilon>0$. For any non-zero ideal $\mathfrak{m}\subset \OO$, we have 
\begin{equation}\label{typeI}
\sum_{\substack{\Norm(\aaa)\leq x\\ (\aaa, F) = 1, \mm\mid \aaa}}s_{\aaa} \ll x^{1-\frac{\delta}{54n|H|^2} + \epsilon},
\end{equation}
where the implied constant depends only on $K$ and $\epsilon$.
\end{prop}

\begin{prop}\label{prop:SumsTypeII}
Let $\epsilon>0$. For any pair of sequences of complex numbers $\{v_{\aaa}\}$ and $\{w_{\bb}\}$ indexed by non-zero ideals in $\OO$ and satisfying $|v_{\aaa}|, |v_{\bb}|\leq 1$,  we have
\begin{equation}\label{typeII}
\sum_{\substack{\Norm(\aaa)\leq x \\ (\aaa, F)=1}}\sum_{\substack{\Norm(\bb)\leq y\\ (\bb, F)=1}}v_{\aaa}w_{\bb}s_{\aaa\bb} \ll \left(x^{-\frac{1}{6n}}+y^{-\frac{1}{6n}}\right)\left(xy\right)^{1+\epsilon},
\end{equation}
where the implied constant depends only on $K$ and $\epsilon$.
\end{prop}

\subsection{Proof of Proposition~\ref{prop:SumsTypeI}}
The proof is very similar to the proof of \cite[(2.5), p.\ 7]{JDS}, so we will outline the additional arguments necessary to prove Proposition~\ref{prop:SumsTypeI}. For each non-zero ideal $\aaa$, there exists a prime ideal~$\pp\in\mathcal{C}$ such that $\aaa\pp^2$ is principal. We can thus write
$$
A_{\mathfrak{m}}(x) = \sum_{\pp\in\mathcal{C}}A_{\mm}(x; \pp),
$$
where 
$$
A_{\mm}(x; \pp) =  \sum_{\substack{\Norm(\aaa)\leq x \\ (\aaa, F) = 1,\ \mathfrak{m} \mid \aaa \\ \aaa\pp^2\text{ is principal}}} s_{\aaa}.
$$
Since $\mathcal{C}$ depends only on $K$, it now suffices to prove that
$$
A_{\mm}(x; \pp) =  \sum_{\substack{\Norm(\aaa)\leq x \\ (\aaa, F) = 1,\ \mathfrak{m} \mid \aaa \\ \aaa\pp^2\text{ is principal}}} s_{\aaa} \ll x^{1-\frac{\delta}{54n|H|^2} + \epsilon}
$$
for each $\pp\in\mathcal{C}$, where the implied constant depends only on $K$ and $\epsilon$. We now use the bijection~\eqref{eq:KeyBijection}, the formula~\eqref{eq:transform}, and the equivalence~\eqref{eq:r0Bijection} to write
$$
A_{\mm}(x; \pp) =  \sum_{\substack{\alpha_0\in \DD,\ \Norm(\alpha_0)\leq x\Norm(\pp)^2 \\ (\alpha_0, F) = 1,\ \alpha_0\equiv 0\bmod [\mm, \pp^2] \\ \alpha_0^h\in A}} \prod_{\sigma\in H}\leg{\alpha_0}{\sigma(\alpha_0)},
$$
where $[\mm, \pp^2]$ denotes the least common multiple of $\mm$ and $\pp^2$. Again, since $\mathcal{C}$ and so also the norms $\{\Norm(\pp)\}_{\pp\in\mathcal{C}}$ depend only on $K$, it suffices to prove that
\begin{equation}\label{eq:boundSumsI}
A'_{\mm}(x) =  \sum_{\substack{\alpha\in \DD,\ \Norm(\alpha)\leq x \\ (\alpha, F) = 1,\ \alpha\equiv 0\bmod \mm \\ \alpha^h\in A}} \prod_{\sigma\in H}\leg{\alpha}{\sigma(\alpha)}\ll_{K, \epsilon} x^{1-\frac{\delta}{54n|H|^2} + \epsilon}
\end{equation}
uniformly for all non-zero ideals $\mm$. We have thus removed the issue of summing terms involving $\spin(\aaa, \sigma)$ for non-principal ideals $\aaa$. It remains to handle the condition $\alpha^h\in A$. To this end, we split the sum into congruence classes modulo $F$, and we emphasize that $F$ is a multiple of $4$. We get
$$
A'_{\mm}(x) = \sum_{\substack{\rho\bmod F \\ \rho\in\Omega_I(A)}}A'_{\mm}(x; \rho),
$$
where
\begin{equation}\label{eq:compareJDS}
A'_{\mm}(x; \rho) = \sum_{\substack{\alpha\in\DD,\ \Norm(\alpha)\leq x \\ \alpha\equiv \rho \bmod F\\ \alpha\equiv 0\bmod \mm}}\prod_{\sigma\in H}\leg{\alpha}{\sigma(\alpha)}
\end{equation}
and where $\Omega_I(A)$ is the set of congruence classes $\rho$ modulo $F$ such that $(\rho, F) = 1$ and such that
$$
\alpha\equiv \rho\bmod F\Longrightarrow \alpha^h\in A.
$$
Note that $|\Omega_I(A)|\leq F$.

The sum $A'_{\mm}(x; \rho)$ in \eqref{eq:compareJDS} is identical to the sum $A(x, \rho)$ in \cite[(3.2), p.\ 9]{JDS}. Hence, the bound for $A(x, \rho)$ proved in \cite[Section 3]{JDS} carries over to $A'_{\mm}(x; \rho)$, which, in conjunction with the fact that $F$ depends only on $K$, implies the bound \eqref{eq:boundSumsI} and hence also Proposition~\ref{prop:SumsTypeI}.

\subsection{Proof of Proposition~\ref{prop:SumsTypeII}}
The proof is very similar to the proof of \cite[(2.6), p.\ 7]{JDS}, so we will outline the additional arguments necessary to prove Proposition~\ref{prop:SumsTypeII}. Given $x, y>0$ and two sequences $v = \{v_{\aaa}\}_{\aaa}$ and $w = \{w_{\bb}\}_{\bb}$ of complex numbers bounded in modulus by $1$, recall that we defined
\begin{equation}\label{BS1}
B(x, y; v, w) = \sum_{\substack{\Norm(\aaa)\leq x\\ (\aaa, F) = 1}}\sum_{\substack{\Norm(\bb)\leq y\\ (\bb, F) = 1}}v_{\aaa}w_{\bb}s_{\aaa\bb},
\end{equation}
and that our goal is to prove that 
\begin{equation}\label{BS2}
B(x, y; v, w) \ll_{K, \epsilon} \left(x^{-\frac{1}{6n}}+y^{-\frac{1}{6n}}\right)\left(xy\right)^{1+\epsilon}
\end{equation}
for all $\epsilon>0$, uniformly in $v$ and $w$. We can write
$$
B(x, y; v, w) = \sum_{\pp_1\in\mathcal{C}}\sum_{\pp_2\in\mathcal{C}}B(x, y; v, w; \pp_1, \pp_2),
$$
where, for $(\pp_1, \pp_2)\in\mathcal{C}\times\mathcal{C}$, we set
$$
B(x, y; v, w; \pp_1, \pp_2) = \sum_{\substack{\Norm(\aaa)\leq x \\ (\aaa, F) = 1\\ \aaa\pp_1^2\text{ is principal}}}\sum_{\substack{\Norm(\bb)\leq y \\  (\bb, F) = 1\\ \bb\pp_2^2\text{ is principal}}}v_{\aaa}w_{\bb}s_{\aaa\bb}.
$$
It suffices to prove the desired estimate for each of the $h^2$ sums $B(x, y; v, w; \pp_1, \pp_2)$. So fix $(\pp_1, \pp_2)\in\mathcal{C}\times\mathcal{C}$. Writing $\pi_1$, $\pi_2$, $\alpha_0$, and $\beta_0$ for the totally positive generators of the principal ideals $\pp_1^h$, $\pp_2^h$, $\aaa\pp_1^2$, and $\bb\pp_2^2$, respectively, we obtain in a similar way to \eqref{eq:transform} the formula
\begin{equation}\label{eq:transform2}
\spin(\aaa\bb, \sigma) = \leg{\alpha_0\beta_0}{\sigma(\alpha_0\beta_0)} = \leg{\alpha_0}{\sigma(\alpha_0)}\leg{\beta_0}{\sigma(\beta_0)}\leg{\alpha_0}{\sigma(\beta_0)\sigma^{-1}(\beta_0)}.
\end{equation}
Using the bijection~\eqref{eq:KeyBijection}, the formula~\eqref{eq:transform2}, and the equivalence~\eqref{eq:r0Bijection}, we deduce that
\begin{equation}
B(x, y; v, w; \pp_1, \pp_2) = \sum_{\substack{\alpha_0\in\DD\\ \Norm(\alpha_0)\leq x\Norm(\pp_1)^2 \\ (\alpha_0, F) = 1 \\ \alpha_0\equiv 0\bmod \pp_1^2}}\sum_{\substack{\beta_0\in\DD\\ \Norm(\beta_0)\leq y\Norm(\pp_2)^2 \\  (\beta_0, F) = 1\\ \beta_0 \equiv 0\bmod \pp_2^2 \\ (\alpha_0\beta_0)^h\in A}}v'_{\alpha_0}w'_{\beta_0}\phi(\alpha_0, \beta_0),
\end{equation}
where
$$
v'_{\alpha_0} = v_{(\alpha_0)/\pp_1^2}\prod_{\sigma\in H}\leg{\alpha_0}{\sigma(\alpha_0)}\quad\text{and}\quad w'_{\beta_0} = w_{(\beta_0)/\pp_2^2}\prod_{\sigma\in H}\leg{\beta_0}{\sigma(\beta_0)}
$$
and where $\phi(\cdot,\cdot)$ is the same function as the one defined in \cite[p.\ 19]{JDS}, i.e., 
$$
\phi(\alpha_0, \beta_0) = \prod_{\sigma\in H}\leg{\alpha_0}{\sigma(\beta_0)\sigma^{-1}(\beta_0)}.
$$
We further split the sum $B(x, y; v, w; \pp_1, \pp_2)$ into congruence classes modulo $F$. As $F$ is divisible by $4$, this will have the effect of separating the variables $\alpha_0$ and $\beta_0$ in the condition $(\alpha_0\beta_0)^h\in A$. We have
$$
B(x, y; v, w; \pp_1, \pp_2) = \sum_{\rho_1\bmod F}\sum_{\substack{\rho_2\bmod F \\ (\rho_1, \rho_2)\in\Omega_{II}(A)}}B(x, y; v, w; \pp_1, \pp_2; \rho_1, \rho_2),
$$
where
$$
B(x, y; v, w; \pp_1, \pp_2; \rho_1, \rho_2) = \sum_{\substack{\alpha_0\in\DD\\ \Norm(\alpha_0)\leq x\Norm(\pp_1)^2 \\ \alpha_0\equiv \rho_1\bmod F}}\sum_{\substack{\beta_0\in\DD\\ \Norm(\beta_0)\leq y\Norm(\pp_2)^2 \\  \beta_0 \equiv \rho_2\bmod F }}v''_{\alpha_0}w''_{\beta_0}\phi(\alpha_0, \beta_0).
$$
Here
$$
v''_{\alpha_0} = \ONE(\alpha_0\equiv 0\bmod \pp_1^2) \cdot v'_{\alpha_0}
$$
and
$$
w''_{\beta} = \ONE(\beta_0\equiv 0\bmod \pp_2^2)\cdot w'_{\beta_0}, 
$$
where $\ONE(P)$ is the indicator function of a property $P$, and $\Omega_{II}(A)$ is the set of $(\rho_1, \rho_2)\in(\OO/(F))^{\times}\times(\OO/(F))^{\times}$ such that 
$$
\alpha_0\equiv \rho_1\bmod F\text{ and }\beta_0\equiv\rho_2\bmod F\Longrightarrow (\alpha_0\beta_0)^h\in A.
$$
Note that $|\Omega_{II}(A)|\leq F^2$.

The sum $B(x, y; v, w; \pp_1, \pp_2; \rho_1, \rho_2)$ has the same shape as the sum $B_i(x, y; \alpha_0, \beta_0)$ in \cite[p.\ 19]{JDS}, and so the bound \cite[(4.5), p.\ 19]{JDS} implies that
$$
B(x, y; v, w; \pp_1, \pp_2; \rho_1, \rho_2)  \ll_{K, \epsilon} \left(x^{-\frac{1}{6n}}+y^{-\frac{1}{6n}}\right)\left(xy\right)^{1+\epsilon}.
$$
This finishes the proof of Proposition~\ref{prop:SumsTypeII} and hence also of Theorem~\ref{mainJointDist}.

\section{Proof of Main Results}

\DRPformulas*

\begin{proof} 
By Theorem \ref{thm:pmdensityformulas}\switchAprint{thm:pmdensityformulas}, $d(R_\pm|S_\pm)= \#\ker(\star_\pm)/2^{(n-1)}$. Then $d(R_\pm|S_\pm)=s_\pm/2^{(n-1)}$ by Proposition \ref{prop:Starlight}\switchAprint{prop:Starlight}. By Theorem \ref{thm:JoinSpinCor}\switchAprint{thm:JoinSpinCor}, $d(F_\pm|R_\pm) = 2^{-(n-1)/2}$. Therefore
\[
    d(F_\pm|S_\pm) =d(F_\pm|R_\pm)d(R_\pm|S_\pm)= \frac{s_\pm}{2^{3(n-1)/2}}.
\]
Since $d(F|S)=d(F_+|S_+)d(S_+|S)+d(F_-|S_-)d(S_-|S)$, and $d(S_\pm|S)=1/2$,
\[
d(F|S)=\frac{s_++s_-}{2^{(3n-1)/2}}.
\qedhere
\]
\end{proof}

Theorem \ref{thm:DRP}\switchAprint{thm:DRP} settles Conjecture~1.1 in \cite{APSR}. This conjecture was originally stated for number fields $K$ which in addition to satisfying properties (C1)-(C4), were also assumed to have prime degree. While as originally stated, this assumption is necessary, it is artificial here. In \cite{APSR}, $m_K$ is defined as the number of non-trivial $\Gal(K/\QQ)$-orbits of $\mathbf{M}_4$ with representative $\alpha\in\OO$ such that $(\alpha,\alpha^\sigma)_2=1$. Let $s$ denote the number of \textit{elements} of $\mathbf{M}_4$ with representative $\alpha\in\OO$ such that $(\alpha,\alpha^\sigma)_2=1$. When $n$ is prime, $s=m_Kn+1$. 

 Let $E$ denote the set of rational primes $p$ such that for $\pp$ a prime of $K$ above $p$, $\spin(\pp,\sigma)=1$ for all non-trivial $\sigma\in\Gal(K/\QQ)$. For a fixed sign let $E_\pm$ denote the set of primes of $E$ congruent to $\pm1\bmod 4$. 
 
 Conjecture~1.1 in \cite{APSR} made two assertions, one regarding the density $d(E|S)$ of such primes restricted to those splitting completely in $K/\QQ$ and one regarding the overall density $d(E)$ of such primes. The assertion regarding the restricted density is correct and the assertion regarding the overall density is slightly off due to a very simple oversight in the case in which $p$ is not assumed to split completely in $K/\QQ$. 
 
\begin{thm}[{\cite[Conjecture~1.1]{APSR}}] Let $K$ be a number field with prime degree satisfying properties (C1)-(C4). Then 
\[
d(E|S) = \frac{s}{2^{(3n-1)/2}}, \quad d(E) = \frac{s}{n2^{(3n-1)/2}},
\]
\[
d(E_\pm|S_\pm) = \frac{s_\pm}{2^{3(n-1)/2}}, \quad \text{and} \quad d(E_\pm) = \frac{s_\pm}{n2^{(3n-1)/2}}.
\]
When $n$ is prime, $s=m_Kn+1$.
\end{thm}

\begin{proof}
If $\pp$ is a prime of $K$ that does not split completely in $K/\QQ$, then for some non-trivial $\sigma\in\Gal(K/\QQ)$, $\pp^\sigma=\pp$ so $\spin(\pp,\sigma)=0$. Therefore $E\subseteq S$ so this $E$ is exactly the $F$ studied in Theorem \ref{thm:DRP}\switchAprint{thm:DRP} and  $E_\pm=F_\pm$.

 Where $s_\pm$ are as given in Theorem \ref{thm:DRP}\switchAprint{thm:DRP}, $s=\#\ker(\star)=s_++s_-$ by Proposition \ref{prop:Starlight}\switchAprint{prop:Starlight}. That $d(E|S)= s/{2^{(3n-1)/2}}$ and $d(E_\pm|S_\pm) = s_\pm/{2^{3(n-1)/2}}$ then follows from Theorem \ref{thm:DRP}\switchAprint{thm:DRP}.

By the Chebotarev Density Theorem, $d(S)=1/n$ and $d(S_\pm)=1/(2n)$ so breaking up the overall density as $d(E)=d(E\cap I|I)d(I)+d(E\cap S|S)d(S)$, we see that $d(E) = s/{n2^{(3n-1)/2}}$. Similarly, $d(E_\pm)={s_\pm}/{n2^{(3n-1)/2}}$.

When $n$ is prime, each non-trivial $\Gal(K/\QQ)$-orbit of $\mathbf{M}_4$ has $n$ elements. As for the trivial orbits, as in \cite[Lemma~5.2]{APSR}, $\star(1)=1$ and $\star(-1)=-1$. Therefore when $n$ is prime, $s=\#\ker(\star)=m_Kn+1$.
\end{proof}

\DRPcubic*

\begin{proof}
For $K$ a cyclic cubic number field with odd class number, by Theorem V in \cite{AF}, all signatures are represented by units.

Two units are equivalent in $\OO^\times/\OO^\times_+$ exactly when they share the same signature. There are $2^{3}$ signatures. Since $(\OO^\times)^2 \subseteq \OO^\times_+$ and $\OO^\times/(\OO^\times)^2 \ism (\ZZ/2)^{3}$, it is therefore the case that all signatures are represented by units exactly when $(\OO^\times)^2$ and $\OO^\times_+$ coincide. It is a consequence of the classical Burgess's inequality \cite{Burgess} that Conjecture~$C_{\eta}$ is true for $m=3$, as is shown in Section 9 of \cite{FIMR}.
Therefore the result follows from  Theorem \ref{thm:DRP}\switchAprint{thm:DRP}.
\end{proof}

\section*{Acknowledgements}

This research was funded by the Max Planck Institute for Mathematics, Cornell University, and ERC grant agreement No.\ 670239. The authors would also like to thank Andrew Granville, Brian Hwang, Christian Maire, Pieter Moree, and Ravi Ramakrishna.

\bibliographystyle{alpha}
\bibliography{DRPbib}

\end{document}